\documentclass[11pt]{amsart}

\usepackage[colorlinks=true,
pdfstartview=FitV, linkcolor=cyan, citecolor=magenta,
urlcolor=blun]{hyperref}
\usepackage{stmaryrd}
\usepackage{enumerate}
\usepackage[pdftex]{graphicx}
\usepackage{tikz} 
\usepackage{amsmath}
\usepackage{amssymb}
\usepackage{mathrsfs}
\usepackage{faktor}
\usepackage{mathdots}
\usepackage{amsthm}
\usepackage{relsize}
\usepackage{bbm}
\usepackage{color}
\usepackage[scale=0.8, top=2.5cm, bottom=2.5cm, left=2.5cm, right=2.5cm]{geometry}

\newtheorem{thm}{Theorem}[section]
\newtheorem*{thm*}{Theorem}
\newtheorem{prop}[thm]{Proposition}
\newtheorem{lem}[thm]{Lemma}
\newtheorem{cor}[thm]{Corollary}

\newtheorem*{que*}{Question}

\theoremstyle{definition}
\newtheorem{defin}[thm]{Definition}
\newtheorem{example}[thm]{Example}

\theoremstyle{remark}
\newtheorem{notation}[thm]{Notation}
\newtheorem{remark}[thm]{Remark}

\newcommand{\bdry}{\partial_\infty}
\newcommand{\veps}{\varepsilon}

\def\bb #1{\mathbb{#1}}

\def\cal #1{\mathcal{#1}}

\def\wt #1{\widetilde{#1}}

\def\rm #1{\mathrm{#1}}
\def\sf #1{\mathsf{#1}}

\numberwithin{equation}{section}

\title[Tree-type]{Sequences of Hitchin representations of Tree-type}
\author{Giuseppe Martone}

\begin{document}
\maketitle
\noindent
\begin{abstract}
The \emph{Hitchin component} is a connected component of the character variety $\sf {Char}_m(S)$ of reductive group homomorphisms from the fundamental group of a closed surface $S$ of genus greater than 1 to the Lie group $\rm {PSL}_m(\bb R)$. The Teichm\"uller space of $S$ naturally embeds into the Hitchin component. The limit points in the Thurston compactification of the Teichm\"uller space are well-understood. Our main goal is to provide non-trivial sufficient conditions on a sequence of Hitchin representations so that the limit of this sequence in the Parreau boundary can be described as a $\pi_1(S)$-action on a tree. These non-trivial conditions are given in terms of Fock-Goncharov coordinates on moduli spaces of generic tuples of flags.
\end{abstract}
\tableofcontents

\section{Introduction}\label{sec:intro}

Let $S$ be a connected, oriented, closed surface of genus $g>1$. The \emph{Teichm\"uller space} $\cal T(S)$ is the space of hyperbolic structures on $S$. The holonomy homomorphism lets us identify the Teichm\"uller space with a connected component of the \emph{character variety}
\[
\sf {Char}_m(S)=\rm {Hom}^*(\pi_1(S),\rm{PSL}_m(\bb R))/\rm{PSL}_m(\bb R)
\]
of conjugacy classes of reductive group homomorphisms from $\pi_1(S)$ to $\rm {PSL}_m(\bb R)$ for $m=2$. 

Hitchin \cite{Hit92} used the theory of Higgs bundles to single out a connected component in $\sf{Char}_m(S)$ for $m\geq 3$ that shares many properties with $\cal T(S)$. This \emph{Hitchin component} $\sf{Hit}_m(S)$ is a real analytic manifold isomorphic to $\bb R^{(m^2-1)(2g-2)}$ containing an embedded copy of the Teichm\"uller space.\footnote{This embedded copy is given by post-composing representations in $\cal T(S)$ with the unique (up to conjugation) irreducible representation of $\rm{PSL}_2(\bb R)$ into $\rm{PSL}_m(\bb R)$.} Furthermore, Fock and Goncharov \cite{FG06} and Labourie \cite{Lab06} showed that Hitchin representations are discrete and faithful. 

Work of Parreau \cite{Par12} provides a compactification of the Hitchin component with boundary points corresponding to actions of $\pi_1(S)$ on a \emph{Bruhat-Tits building} $\cal B_m$. This is a generalization of \cite{MoSh84, Bes88, Pau88}, where the boundary points in Thurston's compactifcation of the Teichm\"uller space \cite{Thu88} are realized as actions of $\pi_1(S)$ on \emph{real trees}.

In recent years, a wide variety of tools have been applied to investigate the boundary of the Hitchin component \cite{Ale08,BIPP17,BIPP19,CoLi17,CDLT,FG06,Le1,Le2,Lof07,M18,Par15b,Par15a}. In this paper, we consider the positivity properties given by \cite{FG06} and single out sequences of Hitchin representations whose asymptotic behavior is similar to the asymptotic behavior of a sequence of representations in the Teichm\"uller space.

We now introduce the background necessary to state our main results. Fock and Goncharov \cite{FG06} used representation theoretic methods to show that a moduli space closely related to the Hitchin component is a \emph{positive variety}: it admits an atlas such that the coordinate changes are quotients of polynomials with only positive coefficients. This positivity has several important consequences for Hitchin representations that we briefly recall in \S\S \ref{ssec:flags}, \ref{ssec:flips}, and \ref{ssec:finitepoly}.

Bonahon and Dreyer \cite{BoDr14,BoDr17} used the (logarithms of the) Fock-Goncharov coordinates to parametrize the Hitchin component. Their parametrization extends Thurston's shearing parametrization of the Teichm\"uller space \cite{Thu86,Bon96}. One starts by choosing a \emph{maximal geodesic lamination $\lambda$}, namely a closed subset of $S$ which is a union of disjoint complete simple curves, called \emph{leaves}, that cut the surface into ideal triangles. The parameters of a Hitchin representation $\rho$ with respect to $\lambda$ come in two flavors: there are the \emph{shearing parameters} $\sigma^a(\rho,e)$, $a=1,\dots, m-1$ associated to a leaf $e$, and there are the \emph{triangle parameters} $\tau^{abc}(\rho, t)$ associated to the ideal triangle $t$. Here, $a,b,c\in\bb Z_{>0}$ with $a+b+c=m$. See \S \ref{ssec:coords} for a more detailed discussion. 

Our main goal is to provide sufficient conditions on the coordinates of a sequence of Hitchin representations so that the limit of this sequence in the Parreau boundary can be described as a $\pi_1(S)$-action on a simplicial tree. 

In order to motivate our sufficient conditions, let us briefly recall the construction of the points in the Parreau boundary. A sequence of Hitchin representations defines a sequence of actions of $\pi_1(S)$ on the symmetric space of $\rm{PSL}_m(\bb R)$. The limiting action of $\pi_1(S)$ on the Bruhat-Tits building $\cal B_m$ arises from taking an \emph{asymptotic cone} of the symmetric space, where we rescale the metric on $\cal B_m$ via an appropriate infinitesimal sequence. 

The symmetric space of $\rm{PSL}_2(\bb R)$ is the hyperbolic plane $\bb H^2$ and its asymptotic cone is a real tree. We can describe this degeneration on (pairs of adjacent) ideal triangles: when rescaling the metric on $\bb H^2$, an ideal triangle converges to a tripod in the limiting tree. Moreover, given two adjacent ideal triangles, one can define the shearing along their common edge. In the limit, these two ideal triangles converge to two tripods $T$ and $T'$. If we are at the center $O_T$ of the tripod $T$, i.e. the intersection of the three lines forming $T$, the relative position of the center $O_{T'}$ of $T'$ is encoded by the asymptotic behavior of the shearing. More precisely, the sign of the rescaled shearing tells us if $O_{T'}$ is to the left or to the right as seen from $O_T$, and the asymptotic behavior of the absolute value of the shearing describes the distance between the two centers.

The analogous description gets more complicated when $m\geq 3$. The case $m=3$ was treated in \cite{Par15b,Par15a}. Using different methods, in \cite{M18} we described explicitly the limits of ideal triangles and ideal quadrilaterals using the Fock-Goncharov parameters for every $m\geq 3$. This allows us to extrapolate explicit conditions on the Fock-Goncharov coordinates that guarantee:
\begin{enumerate}[(A)]
	\item the limit of an ideal triangle has a natural center.
	\item two adjacent ideal triangles have a well-defined notion of shearing \emph{to left or to the right}.
\end{enumerate}

To be more precise, we start by fixing a maximal geodesic lamination $\kappa$ obtained by first cutting the surface $S$ into a family of pair of pants $\cal P$ and then in each pair of pants choosing leaves as in Figure \ref{fig:stdgeod}. This is done in order to assign coordinates to the Hitchin component.

\begin{defin}[Definition \ref{def:treetype}] A sequence of Hitchin representations $(\rho_n)$ is of \emph{tree-type with respect to the maximal geodesic lamination $\kappa$} if there exists a sequence of strictly positive real numbers $(r_n)_n\subset \bb R$ converging to 0, and such that
\begin{itemize}
	\item[(A)] for every ideal triangle $t$ defined by $\kappa$, and for every $a,b,c\in\bb Z_{>0}$ with $a+b+c=m$
	\[
	\lim_{n\to\infty} r_n \tau^{abc}(\rho_n,t)=0,
	\]
	\item[(B)] for every leaf $e$ contained in a pair of pants for $\kappa$, the limit 
	\[
	\lim_{n\to\infty} r_n\sigma^a(\rho_n, e)
	\] 
	is either
	\begin{itemize}
	\item[(B1)] non-negative for all $a=1,\dots, m-1$, or
	\item[(B2)] non-positive for all $a=1,\dots, m-1$.
	\end{itemize}
\end{itemize}
We refer to $r_n$ as a \emph{scaling sequence} for $\rho_n$. 
\end{defin}

We say that the scaling sequence is \emph{trivial} if all the limits in the definition of tree-type are zero. We restrict our attention to non-trivial scaling sequences.

\begin{remark} If we consider a sequence of Hitchin representations as a sequence of actions of $\pi_1(S)$ on the symmetric space of $\rm{PSL}_m(\bb R)$, natural choices of scaling sequences are described in \cite{Bes88,Pau88,Par12}. Another point of view is discussed in Appendix \ref{sec:scaling}. See also Remark \ref{rmk:related}.
\end{remark} 

Embedded in the definition of tree-type is an a priori choice of a maximal geodesic lamination. Our first result states that we can find a \emph{preferred maximal geodesic lamination} adapted to the sequence of Hitchin representations of tree-type and to the pants decomposition $\cal P$. 

\begin{thm}[Theorem \ref{thm:FLP}] Let $(\rho_n)$ be a sequence of Hitchin representations of tree-type with respect to the maximal geodesic lamination $\kappa$. Then, there exists a maximal geodesic lamination $\lambda^+$ which extends the pants decomposition $\cal P$ such that for every leaf $e$ contained in a pair of pants and for every ideal triangle $t$ we have
\begin{align*}
\lim_{n\to\infty}r_n\tau^{abc}(\rho_n,t)=0,\\
\lim_{n\to\infty}r_n\sigma^a(\rho_n,e)\geq 0.
\end{align*}
\end{thm}

One could interpret Theorem \ref{thm:FLP} as a generalization of \cite[\S 8.1]{FLP}, as we find a maximal geodesic lamination $\lambda^+$ extending a pair of pants $P$ and a sequence $\rho_n$ of Hitchin representations of tree-type. The main tools needed for the proof of Theorem \ref{thm:FLP} are the \emph{Bonahon-Dreyer length relations} and the connection between the Fock-Goncharov coordinates and cluster algebras. 

We use Theorem \ref{thm:FLP} to describe the action of $\pi_1(S)$ on a simplicial tree. It follows from \cite{FG06,Lab06} that for every Hitchin representation $\rho$ and every conjugacy class $c$ of elements in $\pi_1(S)$, we can define the \emph{Hilbert length of $c$} as
\[
\ell_H(\rho,c)=\log\left(\frac{\text{largest eigenvalue of }\rho(c)}{\text{smallest eigenvalue of }\rho(c)}\right)>0.
\]
We use the preferred maximal geodesic lamination $\lambda^+$ to study the asymptotic behavior of Hilbert length of curves on the surface. For general $m\geq 3$ we can prove the following. Let us denote by $\wt \lambda^+$ the lift of $\lambda^+$ to the universal cover of $S$.

\begin{thm}[Theorem \ref{thm:action}.\ref{cond:anyd}]Let $(\rho_n)$ be a sequence of Hitchin presentations of tree-type and let $\lambda^+$ be the maximal geodesic lamination given by Theorem \ref{thm:FLP}. Let $c$ be a conjugacy class in $\pi_1(S)$ such that
\[
\lim_nr_n\ell_H(\rho_n,c)\in [0,\infty).
\]
and the curve $c$ lies in a pair of pants $P\in\cal P$. Then, there exists a finite set of leaves $\cal E\subset \wt \lambda^+$, which depends on $c$, such that
\[
\lim_nr_n\ell_H(\rho_n,c)=\lim_nr_n \sum_{e\in\cal E}\sigma(\rho_n,e),
\]
where $\sigma(\rho_n,e)$ is the sum of the shearing $\sigma^1(\rho_n,e)+\dots+\sigma^{m-1}(\rho_n,e)$. 
\end{thm}

In the case of $m=3$, we can prove a stronger result. Thanks to \cite{FG06,Lab06}, one can refine the notion of length of a curve. The \emph{simple root lengths} of $\rho$ and $c$ are
\[
\ell_1(\rho,c)=\log\left(\frac{\text{largest eigenvalue of }\rho(c)}{\text{middle eigenvalue of }\rho(c)}\right)>0,\text{ and } \ell_2(\rho,c)=\log\left(\frac{\text{middle eigenvalue of }\rho(c)}{\text{smallest eigenvalue of }\rho(c)}\right)>0
\]

\begin{thm}[Theorem \ref{thm:action}.\ref{cond:d3}] Let $(\rho_n)$ be a sequence of Hitchin presentations of tree-type and let $\lambda^+$ be the maximal geodesic lamination given by Theorem \ref{thm:FLP}. Let $c$ be a conjugacy class in $\pi_1(S)$ such that
\[
\lim_nr_n\ell_H(\rho_n,c)\in [0,\infty).
\]
Assume that $\lim_nr_n\sigma^a(\rho_n,d)\geq 0$ for all $a=1,2$, and $d$ a curve in $\cal P$, and, that there exists $C>0$ such that $\ell_a(\rho,d)>C$ for all $a=1,2$, and $d$ in $\cal P$. Then, there exists a finite set of leaves $\cal E\subset \wt \lambda^+$, which depends on $c$, such that
\[
\lim_nr_n\ell_H(\rho_n,c)=\lim_nr_n \sum_{e\in\cal E}\sigma(\rho_n,e),
\]
where $\sigma(\rho_n,e)=\sigma^1(\rho_n,e)+\dots+\sigma^{m-1}(\rho_n,e)$.
\end{thm}

The key result needed to establish Theorem \ref{thm:action}.\ref{cond:d3} is Theorem \ref{thm:d3} which lets us use the tree-type condition to control the asymptotic behavior of certain parameters nearby the boundary of a pair of pants. It is not clear to us if a similar result holds for dimension $m>3$. 

Theorem \ref{thm:action} is proved by finding a well-adapted totally positive representative for $\rho(c)$ in $\rm{GL}(d,\bb R)$, following \cite{FG06}.

\begin{remark}\label{rmk:related} In related work, Burger, Iozzi, Parreau, and Pozzetti \cite{BIPP17, BIPP19} study our main question in the context of \emph{geodesic currents}.

In \cite{Bon88}, Bonahon described representations in the Teichm\"uller space (and their limits) in terms of certain measures on the space of unoriented geodesics on the universal cover of $S$. Generalizing this work, Zhang and the author \cite{MZ19} associate to each Hitchin representation a \emph{Hilbert geodesic current}, which is well-adapted to the study of the Hilbert length. See also \cite{BCLS18} for an independent construction.

By compactness of the space of projectivized geodesic currents, there exists a scaling sequence $r_n$ such that a sequence of Hilbert geodesic currents scaled by $r_n$ converges to a geodesic current $\mu$. 

Burger, Iozzi, Parreau, Pozzetti \cite{BIPP17, BIPP19} give criteria involving the asymptotic behavior of the length of the shortest curve on the surface that guarantee that a sequence of Hitchin representations converges to a point in the boundary of the Teichm\"uller space.

In Appendix \ref{sec:scaling}, we show that the scaling sequence $r_n$ given by this construction is infinitesimal and, therefore, can be used to define sequences of Hitchin representations of tree-type. The result in Appendix \ref{sec:scaling} was proved by the author together with C. Ouyang. Independently, M. B. Pozzetti communicated the proof to the author.
\end{remark}

\section*{Acknowledgments}

The author wishes to thank Francis Bonahon for the many insightful conversations, above all at the early stages of this project. This work has greatly benefitted from conversations with Marc Burger, Richard Canary, Charles Ouyang, and Maria Beatrice Pozzetti. We take this opportunity to thank them. Part of this work was completed while the author was visiting Rice University. We thank Mike Wolf and the Mathematics Department for their kind hospitality. This research was partially supported by the grant DMS-1406559 from the U.S. National Science Foundation. In addition, the author gratefully acknowledges support from the AMS-Simons travel grant, and NSF grants DMS-1107452, 1107263 and 1107367 ``RNMS: GEometric structures And Representation varieties'' (the GEAR Network).

\section{Coordinates for the Hitchin component}\label{sec:prelim}

Bonahon and Dreyer \cite{BoDr14, BoDr17}, building on work of Fock and Goncharov \cite{FG06}, give a parametrization of the Hitchin component which depends on a choice of a \emph{maximal geodesic lamination}. This construction generalizes the Thurston's shearing parametrization of the Fuchsian space \cite{Thu86, Bon96}. In this section we recall some of the main results from \cite{FG06,BoDr14}.

We start by studying the space $\sf {F}_m(\bb R)$ of flags in $\bb R^m$, which is the space of maximal nested chains of proper vector subspaces of $\bb R^m$. Concretely, a flag $F\in\sf {F}_m(\bb R)$ is a collection of $m-1$ vector subspaces $F^{(a)}\subset \bb R^m$ with $\dim F^{(a)}=a$ and such that $F^{(a)}\subset F^{(a+1)}$. 
\begin{defin} The $k$-tuple of flags $(F_1,\dots F_k)$ is \emph{generic} if for every choice of non-negative integers $a_1,\dots,a_k$ one has
\[
\dim \big(F_1^{(a_1)}+\dots+F_k^{(a_k)}\big)=\min\{a_1+\dots +a_k,m\}.
\]
Let $\sf F^{(k)}_m(\bb R)$ be the space of generic $k$-tuples of flags considered up to the natural action of $\rm{PGL}_m(\bb R)$.
\end{defin}
Our interest in the study of generic tuples of flags is mainly motivated by Theorem \ref{thm:flagmap}, which is one of the main tools for the Bonahon-Dreyer parametrization. Before we state Theorem \ref{thm:flagmap}, let us fix a connected, closed, oriented topological surface $S$ of genus $g>1$. Denote by $\bdry \wt S$ the boundary of the universal cover $\wt S$ of the surface.

\begin{thm}[Theorem 1.4 \cite{Lab06}]\label{thm:flagmap}
Let $\rho\in\sf {Hit}_m(S)$ be a Hitchin representation. There exists a continuous $\rho$-equivariant, well-defined up to $\rm{PGL}_m(\bb R)$-action, \emph{limit map}
\[
\xi_{\rho}\colon\bdry \wt S \to \sf {F}_m(\bb R)
\] 
such that
\begin{enumerate}
	\item[-] the $k$-tuple of flags $(\xi_\rho(x_1),\dots,\xi_\rho(x_k))$ is generic for every $k$-tuple of pairwise distinct points $x_1,\dots, x_k\in\bdry \wt S$;
	\item[-] for every $a_1,\dots, a_k\in \bb Z_{>0}$ with $a=\sum_i a_i\leq m$, for all $x\in\bdry \wt S$, and for sequences $y_1,\dots, y_k$ of pairwise distinct points in $\bdry \wt S$ 
	\[
	\xi_\rho(x)^{(a)}=\lim_{(y_1,\dots, y_k)\to x}\bigoplus_i \xi_\rho(y_i)^{(a_i)}.
	\]
\end{enumerate}
\end{thm}

\subsection{Moduli spaces of flags: coordinates}\label{ssec:flags}

In this section we define parameters for elements of $\sf {F}_m^{(k)}(\bb R)$. We start by fixing, once and for all, an identification $\wedge^m \bb R^m\cong \bb R$.

\begin{remark} When clear from context, we blur the distinction between a generic $k$-tuple of flags and its orbit in $\sf F^{(k)}_m(\bb R)$. We denote both of these objects by $(F_1,\dots, F_k)$.
\end{remark}

\begin{defin}[triple and double ratios] For $F_i\in\sf F_m(\bb R)$ and $\alpha=1,\dots, m-1$, choose an arbitrary non-zero vector $f_i^\alpha$ in the one-dimensional vector subspace $\wedge^\alpha F_i^{(\alpha)}$ of $\wedge^\alpha \bb R^m$. 
\begin{enumerate}[$(i)$]
	\item For $a,b,c\in\bb Z_{>0}$ such that $a+b+c=m$, the \emph{$abc$-triple ratio of $(F_1,F_2,F_3)$ in $\sf F_m^{(3)}(\bb R)$} is
\[
T^{abc}(F_1,F_2,F_3):=
\frac{ f^{a+1}_1\wedge  f^{b}_2\wedge  f^{c-1}_3}{ f^{a-1}_1\wedge  f^{b}_2\wedge  f^{c+1}_3}\cdot
\frac{ f^{a}_1\wedge  f^{b-1}_2\wedge  f^{c+1}_3}{ f^{a}_1\wedge  f^{b+1}_2\wedge  f^{c-1}_3}\cdot
\frac{ f^{a-1}_1\wedge  f^{b+1}_2\wedge  f^{c}_3}{ f^{a+1}_1\wedge  f^{b-1}_2\wedge  f^{c}_3}.
\]
	\item For $a=1,\dots, m-1$, the \emph{$a$-double ratio of $(F_1,F_2,F_3,F_4)\in\sf {F}_m^{(4)}(\bb R)$} is
\[
D^a(F_1,F_2,F_3,F_4)=-\frac{f_1^{m-a-1}\wedge f_3^{a}\wedge f_2^1}{f_1^{m-a-1}\wedge f_3^{a}\wedge f_4^1}\cdot
\frac{f_1^{m-a}\wedge f_3^{a-1}\wedge f_4^1}{f_1^{m-a}\wedge f_3^{a-1}\wedge f_2^1}
\]
\end{enumerate}
\end{defin}
The triple and double ratios are well-defined non-zero real numbers because they are constant on $\rm {PGL}_m(\bb R)$-orbits of $k$-tuple of flags, they do not depend on the choice of the $f^\alpha_i$'s, and the $k$-tuples of flags are generic. The symmetries of triple and double ratios under permutations of flags are
\begin{equation}\label{eqn:symmetries}
\begin{aligned}
T^{abc}(F_1,F_2,F_3)=T^{bca}(F_2,F_3,F_1)=(T^{bac}(F_2,F_1,F_3))^{-1},\\
D^a(F_1,F_2,F_3,F_4)=D^{m-a}(F_3,F_4,F_1,F_2)=(D^a(F_1,F_4,F_3,F_2))^{-1}.
\end{aligned}
\end{equation}

Triple and double ratios parameters for a generic $k$-tuple $(F_1,\dots, F_k)$ in $\sf {F}_m^{(k)}(\bb R)$ with $k\geq 4$ can be defined by fixing additional topological data. 

For example, in the case $k=4$, we can keep track of the lack of symmetry in the formula for the double ratios, by saying that \emph{we labeled the vertices of a quadrilateral by the flags $F_1,F_2,F_3,F_4$, in this cyclic counter-clockwise order, and we chose the oriented diagonal from $F_3$ to $F_1$.}

Let us explain the general procedure. Consider $k$ distinct points $x_1,\dots, x_k$ in this cyclic counter-clockwise order along the unit circle $S^1$. Let $\Pi$ denote the polygon inscribed in $S^1$ resulting from drawing chords between consecutive $x_i$'s. 

An \emph{(ideal) triangulation of $\Pi$} is a maximal set $\lambda=\{e_1,\dots, e_{k-3}\}$ of \emph{oriented} diagonals in $\Pi$. Denote by $\cal T_\lambda$ the set of $k-2$ connected components of $\Pi-\lambda$.

\begin{notation}\label{not:vertices}\
\begin{enumerate}[-] 
	\item Each oriented diagonal $e\in\lambda$, singles out two adjacent triangles in $\cal T_\lambda$ with vertices $x_{e^+},x_{e^l},x_{e^-}$ and $x_{e^+},x_{e^-},x_{e^r}$, respectively. Here, by convention, $x_{e^+}$ is the forward endpoint of $e$ and $x_{e^+}, x_{e^l}, x_{e^-}, x_{e^r}$ appear in this cyclic counter-clockwise order along $S^1$. 
	\item Given a triangle $t\in \cal T_\lambda$ with a preferred vertex $x_{t}$, we label the remaining vertices of $t$ by $x_{t'},x_{t''}$ so that $x_{t},x_{t'},x_{t''}$ appear in this cyclic (counter-clockwise) order.
\end{enumerate}
\end{notation}

For every choice of triangulation $\lambda=\{e_1,\dots, e_{k-3}\}$, with $\cal T_\lambda=\{t_1,\dots, t_{k-2}\}$, define a map 
\[
\phi_\lambda\colon\sf {F}_m^{(k)}(\bb R)\to \bb R^{(k-3)(m-1)} \times \bb R^{(k-2)\binom{m-1}{2}}
\] 
by recording
\begin{enumerate}
	\item[-] the double ratios $D^a(F_{e_i^+}F_{e_i^l},F_{e_i^-},F_{e_i^r})$, for every edge $e_i\in\lambda$, and
	\item[-] the triple ratios $T^{abc}(F_{t_i},F_{t'_{i}},F_{t''_{i}})$, for every triangle $t_i\in \cal T_\lambda$, with preferred vertex $x_{t_i}$.
\end{enumerate}

\begin{remark}\label{rmk:tripleratio} The triple ratios of a triangle $t\in \cal T_\lambda$ as defined above depend on a choice of a preferred vertex of $t$. Any two such choices are related by Equation \ref{eqn:symmetries}. 
\end{remark}

\begin{thm}[Theorem 9.1 \cite{FG06}]\label{thm:posflags} The space
\[
\sf F^{(k+)}_m=
\{(F_1,\dots, F_k)\in \sf F_m^{(k)}(\bb R)\colon \text{all the coordinates of }\phi_{\lambda}(F_1,\dots, F_k)\text{ are positive}\}
\]
is independent of the choice of the triangulation $\lambda$.
\end{thm}

Equip $\bdry \wt S$ with the cyclic counter-clockwise order given by the orientation on $S$.

\begin{thm}[Theorem 1.8 \cite{FG06}, see also Lemma 8.4.2 \cite{LabMcS09}]\label{thm:poslimmap} Let $\rho$ be a Hitchin representation with limit map $\xi_\rho$.
If $(x_1,\dots, x_k)$ appear in this cyclic counter-clockwise order along $\bdry \wt S$, then 
\[
(\xi_\rho(x_1),\dots,\xi_\rho(x_k))\in\sf {F}_m^{(k+)}.
\]
\end{thm}

We refer to $\sf {F}_m^{(k+)}$ as the space of \emph{positive} $k$-tuple of flags. In light of Theorem \ref{thm:poslimmap}, we say that the limit map a Hitchin representation is \emph{positive}.

\subsection{Moduli spaces of flags: diagonal flips}\label{ssec:flips}

Theorem \ref{thm:posflags} states that the space of positive $k$-tuples of flags is well-defined. However, the Fock-Goncharov parameters of a $k$-tuple in $\sf {F}_m^{(k+)}$ depends on the choice of triangulation $\lambda$. In this section, we recall the description of a special family of changes of coordinates for $\sf {F}_m^{(k+)}$, called \emph{diagonal flips}, that is relevant for our purposes.

Following Notation \ref{not:vertices}, an oriented diagonal $e$ in $\lambda$ singles out a cyclically ordered quadruple $(x_{e^+}, x_{e^l},x_{e^-},x_{e^r})$ of vertices of $\Pi$. A triangulation $\kappa$ \emph{differs from $\lambda$ by a diagonal flip at $e$} if $f\in\kappa$ is an oriented diagonal with 
\[
(f^+,f^l,f^-,f^r)=(e^l,e^-,e^r,e^+).
\]
and $\kappa$ is equal to $\lambda$ otherwise. See Figure \ref{fig:flip}. 
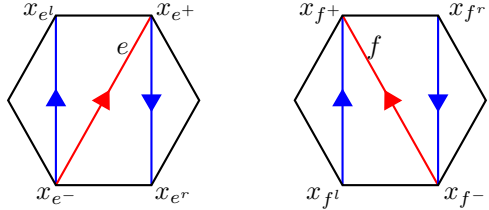
\begin{figure}[h]
\centering
        \includegraphics[scale=.5]{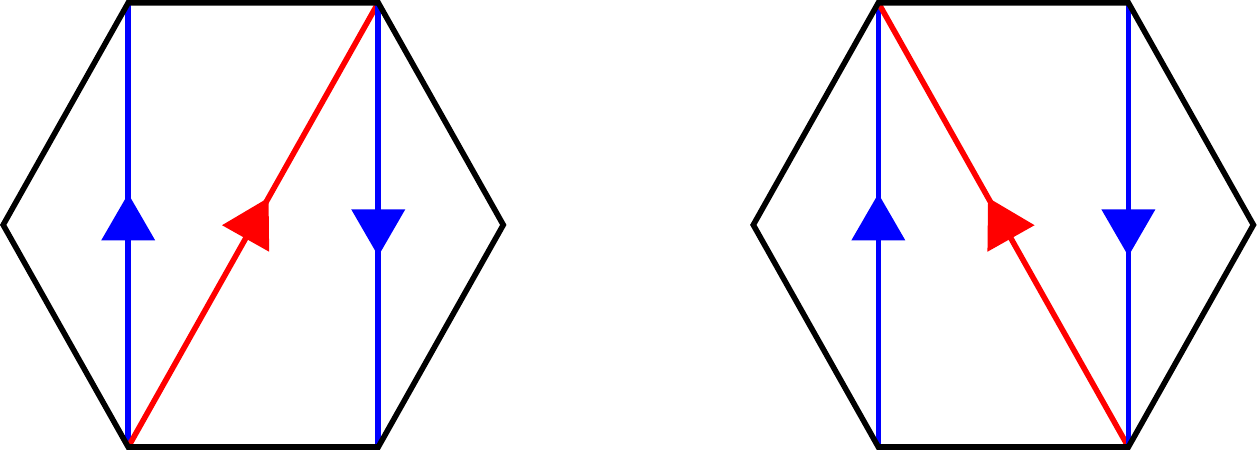}  
        \put (-440, 65){\makebox[0.7\textwidth][r]{\small{$x_{e^+}$}}}
        \put (-493, 65){\makebox[0.7\textwidth][r]{\small{$x_{e^l}$}}}
    	\put (-485, -5){\makebox[0.7\textwidth][r]{\small{$x_{e^-}$}}}
       	\put (-442, -5){\makebox[0.7\textwidth][r]{\small{$x_{e^r}$}}}
	\put (-466, 50){\makebox[0.7\textwidth][r]{\small{$e$}}}
	\put (-370,50){\makebox[0.7\textwidth][r]{\small{$f$}}}
	\put (-385, 65){\makebox[0.7\textwidth][r]{\small{$x_{f^+}$}}}
        \put (-385, -5){\makebox[0.7\textwidth][r]{\small{$x_{f^l}$}}}
    	\put (-330, -5){\makebox[0.7\textwidth][r]{\small{$x_{f^-}$}}}
       	\put (-330, 65){\makebox[0.7\textwidth][r]{\small{$x_{f^r}$}}}
        \caption{Two triangulations of the hexagon that differ by a diagonal flip at $e$.}
        \label{fig:flip}
\end{figure}

We describe the change of the Fock-Goncharov parameters of a positive $k$-tuple of flags under a diagonal flip. More precisely, fix $(F_1,\dots, F_k)\in\sf F^{(k+)}_m$.  Let $f$ be the diagonal in $\kappa$ resulting from diagonal flip of the triangulation $\lambda$ at a diagonal $e$. Our goal for the reminder of this section is to express
\begin{gather*}
D^a(F_{f^+},F_{f^l},F_{f^-},F_{f^r}),\ \ T^{abc}(F_{f^+},F_{f^l},F_{f^-}),\ \ T^{abc}(F_{f^+},F_{f^-},F_{f^r}).
\end{gather*} 
in terms of 
\begin{gather*}
D^a(F_{e^+},F_{e^l},F_{e^-},F_{e^r}),\ \ T^{abc}(F_{e^+},F_{e^l},F_{e^-}),\ \ T^{abc}(F_{e^+},F_{e^-},F_{e^r}).
\end{gather*}
Let us start by considering the case of a positive quadruple $(F_1,F_2,F_3,F_4)$ corresponding to a quadrilateral with vertices $(x_1,x_2,x_3,x_4)$. Let $\lambda=\{e\}$ be the triangulation given by the choice of the diagonal from $x_3$ to $x_1$, and let $\kappa=\{f\}$ be the triangulation given by the choice of the diagonal from $x_4$ to $x_2$. Fock and Goncharov \cite[\S 9]{FG06} describe relations between the parameters defined via $\lambda$ and $\kappa$ by applying cluster algebras methods that we now recall. 

Consider the set
\[
\cal Q_0:=\left\{(a,b,c,d)\in\bb Z^4_{\geq 0}\,\middle\vert\,\begin{array}{c} (a,b,c,d)=(a,b,c,0),\ a+b+c=m\\
\text{or}\\
 (a,b,c,d)=(a,b,0,d),\ a+b+d=m\\
\end{array}\right\}.
\]
One can think of $\cal Q_0$ as the parallelogram obtained from two copies of the discrete triangle 
\[
\{(a,b,c)\in\bb Z^4_{\geq 0}\colon a+b+c=m\}
\] 
glued along the edge composed by the vertices $v=(a,b,0,0)$. An \emph{interior vertex of $\cal Q_0$} is a vertex $(a,b,c,d)$ with $a,b\in \bb Z_{>0}$. Denote by $\rm{int}(\cal Q_0)$ the set of interior vertices of $\cal Q_0$. See Figure \ref{fig:pi}.

\begin{figure}[h]
\begin{minipage}{0.50\textwidth}
        \centering
        \includegraphics[scale=.4]{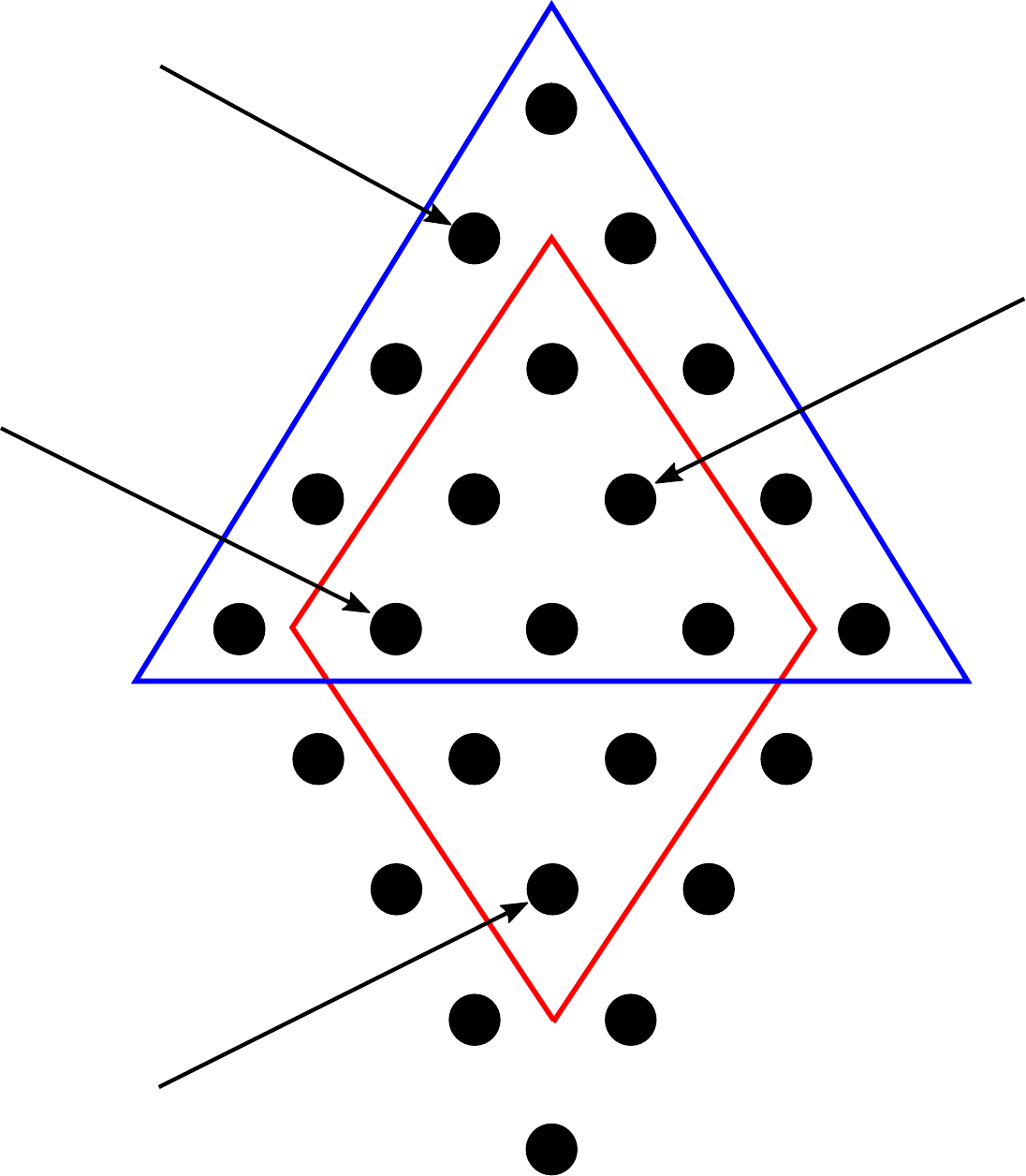}
	\put (-270, 95){\makebox[0.7\textwidth][r]{\tiny{$(1,3,0,0)$}}}
	\put (-250,140){\makebox[0.7\textwidth][r]{\tiny{$(0,1,3,0)$}}}
	\put (-150,110){\makebox[0.7\textwidth][r]{\tiny{$(2,1,1,0)$}}}
	\put (-260,5){\makebox[0.7\textwidth][r]{\tiny{$(1,1,0,2)$}}}
\end{minipage}\hfill
\begin{minipage}{0.50\textwidth}
        \centering
        \includegraphics[scale=.4]{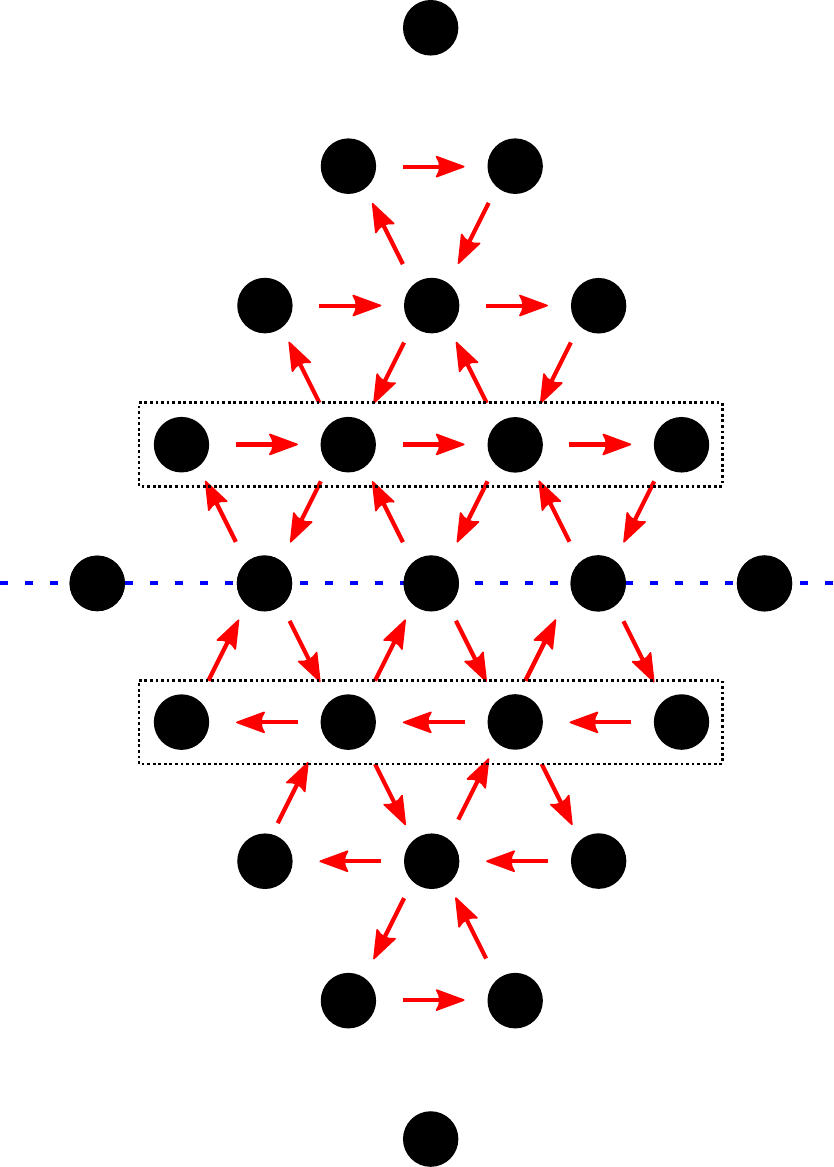}
\end{minipage}
        \caption{On the left hand side, the set $\cal Q_0$ for $m=4$. Highlighted vertices of the form $(a,b,c,0)$ and the interior of $\cal Q_0$. On the right hand side, the directed graph $\cal Q_0(e)$ for $m=4$. The dotted line strikes through the vertices $(a,b,0,0)$ with $a,b\geq 0$. Highlighted the vertices $v=(a,b,c,d)$ with $c+d=1$.}
        \label{fig:pi}
\end{figure}

The choice of $\lambda$ allows us to define double and triple ratios for $(F_1,F_2,F_3,F_4)=(F_{e^+},F_{e^l},F_{e^-},F_{e^r})$. These functions assign positive weights to the interior vertices $v\in\text{Int}(\cal Q_0)$ by setting
\[
X^v=
\begin{cases}
T^{abc}(F_{e^+},F_{e^l},F_{e^-}) &\text{ if }v=(a,b,c,0),\, c\neq 0\\\
T^{abd}(F_{e^+},F_{e^-},F_{e^r}) &\text{ if }v=(a,b,0,d),\, d\neq 0\\ 
D^a(F_{e^+},F_{e^l},F_{e^-},F_{e^r}) &\text{  if }v=(a,b,0,0)
\end{cases}
\]

Consider oriented arrows connecting vertices in $\cal Q_0$ as in the left hand side of Figure \ref{fig:mutation}. Let us denote by $\cal Q_0(e)$ the resulting weighted directed graph. We now describe how to obtain the weighted directed graph $\cal Q_0(f)$ defined by $\kappa$ starting from $\cal Q_0(e)$.
\begin{defin}\label{def:vmutation}
Given a vertex $v\in \rm{int}(\cal Q_0)$, the \emph{mutation} $\mu_v$ is an operation on the weighted directed graph $\cal Q_0(e)$ described by the following sequence of steps. 
\begin{enumerate}[]
	\item \textbf{Step 1:} Do not change the set of vertices $\cal Q_0$.
	\item \textbf{Step 2:} Change the weights according to the rule
	$
\mu_v(X^w)= 
\begin{cases}
1/X^v&\text{ if }v=w\\
X^w(1+X^v) &\text{ if }v\to w\\
X^w\frac{X^v}{(1+X^v)} &\text{ if }w\to v\\
X^w &\text{ otherwise}
\end{cases}
$
	\item \textbf{Step 3:} Change arrows according to Figure \ref{fig:mutation}.
	\item \textbf{Step 4:} Erase any 2-loop, i.e. paths of the form $v\to w\to v$, resulting from Step 3.
\end{enumerate}
\end{defin}

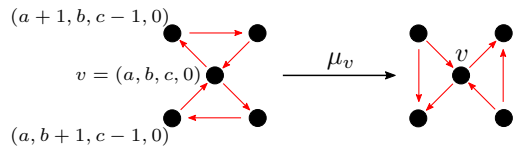
\begin{figure}[h]
\includegraphics[scale=.2]{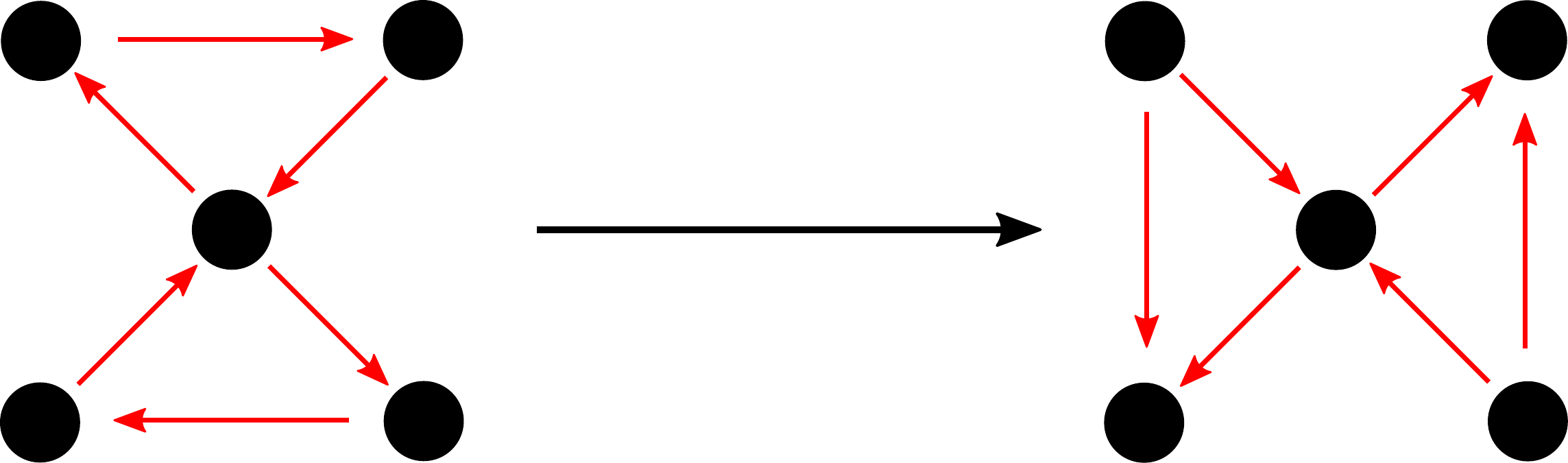} 
\put (-448, 18){\makebox[0.7\textwidth][r]{\tiny{$v=(a,b,c,0)$}}}
\put (-460, -5){\makebox[0.7\textwidth][r]{\tiny{$(a,b+1,c-1,0)$}}}
\put (-460, 40){\makebox[0.7\textwidth][r]{\tiny{$(a+1,b,c-1,0)$}}}
\put (-390,24){\makebox[0.7\textwidth][r]{\small{$\mu_v$}}}
\put (-347, 25){\makebox[0.7\textwidth][r]{\small{$v$}}}
\caption{The change of arrows for a mutation $\mu_v$.}
\label{fig:mutation}
\end{figure}

The diagonal flip from $\cal Q_0(e)$ to $\cal Q_0(f)$ is the result of an explicit composition of mutations. 

For $s=0,1,\dots, m-2$, the \emph{$s$-stratum} is the set of vertices $v=(a,b,c,d)$ such that $c+d=s$. 
An \emph{$s$-stratum mutation $\bar\mu_s$} is the composition of $v$-mutations at all the interior vertices $v$ of the $s$-stratum. Note that $\bar\mu_s$ is independent on the order of composition.

The diagonal flip is a combination of (restrictions of) strata mutations. Consider the nested sets
\[
\cal Q_{r}=\{(a,b,c,d)\in\cal Q_0\colon a,b\geq r\},\text{and } \rm{int}(\cal Q_{r})=\{(a,b,c,d)\in\cal Q_0\colon a,b>r\}
\]
for $r=1,\dots, r_{\max}$, where $r_{\max}=\lfloor \frac{m}{2}\rfloor$. See Figure \ref{fig:stratamutation1}. 

For $r=1,\dots, r_{\max}$, define iteratively the weighted directed graph $\cal Q_{r}(e)$ with
\begin{enumerate}
	\item[-] set of vertices equal to $\cal Q_{r}$,
	\item[-] arrows resulting from applying $\nu_{r-1}=\bar\mu_{m-2r}\circ\dots\circ\bar \mu_0$ to the arrows of $\cal Q_{r-1}(e)$, and then removing any arrows between two vertices not in $\rm{int}(\cal Q_{r})$,
	\item[-] weights resulting from applying $\nu_{r-1}$ to the weights of $\cal Q_{r-1}(e)$.
\end{enumerate} 
When clear from context, we blur the distinction between $\nu_r$ and its trivial extension to $\cal Q_0(e)$.

\begin{figure}[h]
\includegraphics[scale=.35]{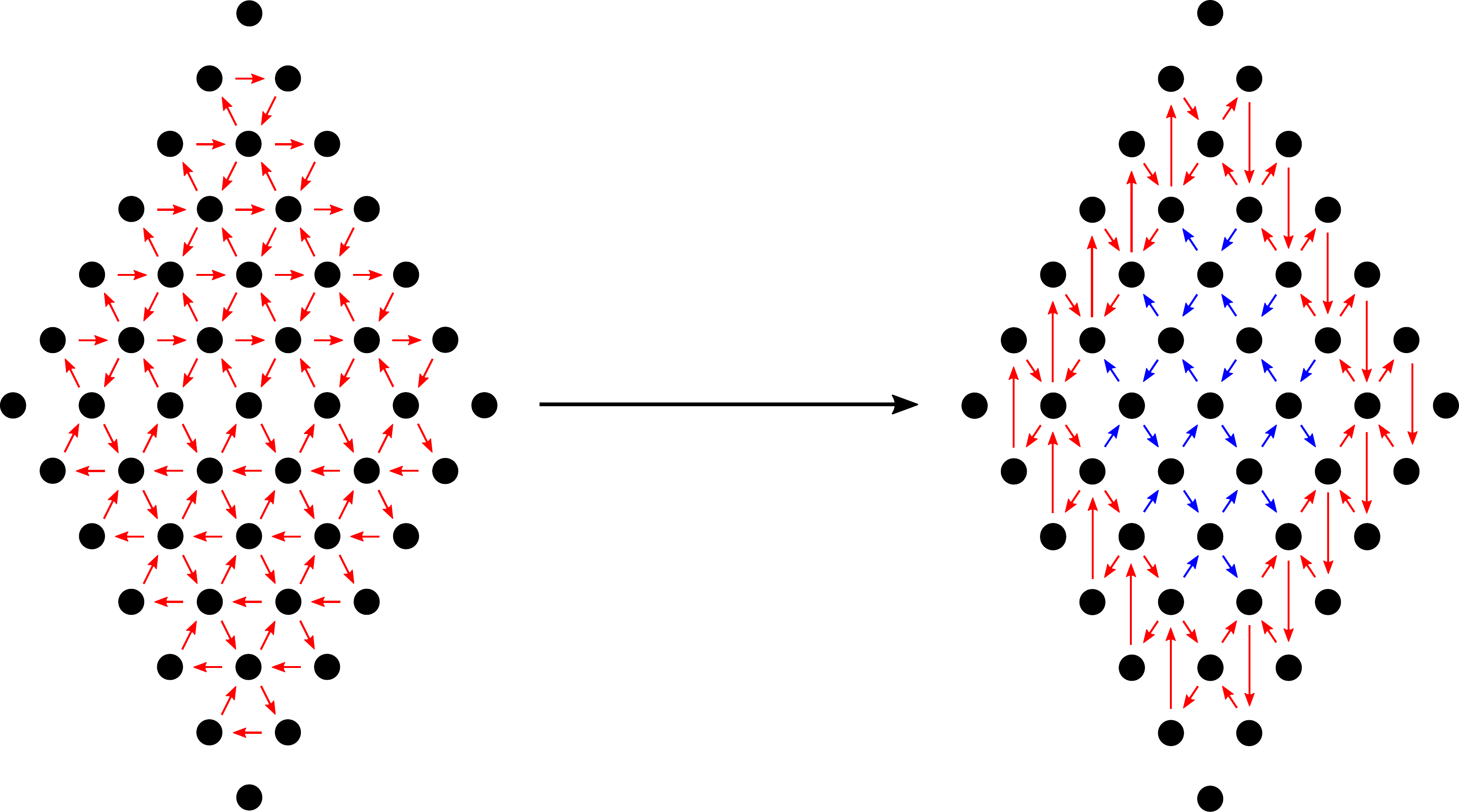}
\put (-480, 90){\makebox[0.7\textwidth][r]{\small{$\nu_0$}}}
\caption{The result of $\nu_{0}$ applied to $\cal Q_{0}(e)$, with $m=6$. On the right, the arrows of the directed graph $\cal Q_{1}(e)$ are highlighted.}
\label{fig:stratamutation1}
\end{figure}
\begin{lem}[\S 10 of \cite{FG06}]\label{lem:strata} With the notations above, the weighted directed graph $\cal Q_0(f)$ is obtained by applying $\nu_{r_{\max}-1}\circ\dots\circ\nu_{0}$ to $\cal Q_0(e)$.
\end{lem}
\begin{remark}\label{rmk:generalflip} Let us generalize the above discussion to the case of a positive $k$-tuple of flags $(F_1,\dots, F_k)$. Assume we have fixed a triangulation $\lambda$ of the polygon $\Pi$ with cyclically ordered vertices $(x_1,\dots,  x_k)$. 

Suppose the flags $F_{e^+},F_{e^l},F_{e^-},F_{e^r}$ correspond to a quadrilateral in $\lambda$ with diagonal $e$ oriented from $x_{e^-}$ to $x_{e^+}$. Then, at least one of the edges of this quadrilateral corresponds to a different diagonal $\bar e$ with double ratios $D^a(\bar e)$, for $a=1,\dots, m-1$.

Let $\kappa$ be the triangulation which differs from $\lambda$ by a diagonal flip at $e$. Note that the $a$th double ratio with respect to $\kappa$ of the diagonal $\bar e$ is, in general, different from $D^a(\bar e)$. This change of coordinates, however, is taken into account by adding weights to the vertices of the edge of $\cal Q_0(e)$ which correspond to $\bar e$. 

For example, assume $\bar e$ is the diagonal from $x_{e^l}$ to $x_{e^+}$. Then, we add weights to $\cal Q_0(e)$ by setting
\[
X^{v}=D^a(\bar e),\text{ for }v=(a,0,m-a,0),\ a=1,\dots, m-1.
\]
Then, after performing the mutations described in Lemma \ref{lem:strata}, the resulting weights for the vertices $v=(a,0,m-a,0)$, and $a=1,\dots, m-1$ coincide with the Fock-Goncharov parameters of the diagonal $\bar e$ with respect to the triangulation $\kappa$.
\end{remark}

\subsection{Parametrization of the Hitchin component}\label{ssec:coords}

In this section we recall Bonahon and Dreyer's parametrization of the Hitchin components in the special case of \emph{maximal geodesic laminations extending a pants decomposition}. We refer to \cite{BoDr14,BoDr17} for the general case, which is, however, not needed for this paper.

\subsubsection{Maximal geodesic laminations extending a pants decomposition}

In the context of Hitchin representations, the choice of topological data described in \S \ref{ssec:flags} is replaced by a choice of a \emph{maximal geodesic lamination}. Maximal geodesic laminations are topological objects (see for example \cite{Thu78,Bon96,PH92}), however it is convenient to endow $S$ with an auxiliary hyperbolic metric when defining them.

We start by fixing once and for all a pants decomposition $\cal P$ of the surface $S$. This gives a collection $\lambda_c=\{c_1,\dots, c_{3g-3}\}$ of closed geodesics on $S$, called \emph{closed leaves}. We equip the $c_i$'s with arbitrary orientations.

A \emph{maximal geodesic lamination $\lambda_{P}$ for $P\in\cal P$} is a set $\lambda_{P}=\{e_1,e_2,e_3\}$ of pairwise non-intersecting, simple, oriented, bi-infinite geodesic which spiral around boundary components of $P$. Denote by $\lambda_o=\cup_{P\in\cal P}\lambda_{P}$. We refer to $e\in\lambda_o$ as an \emph{open leaf}.

The union $\lambda=\lambda_c\cup\lambda_o$ is a \emph{maximal geodesic lamination of $S$ (extending $\cal P$)}. Denote by $\cal T_{\lambda,P}$ the set of connected components of $P-\lambda_{P}$ for $P\in\cal P$, and by $\cal T_\lambda$ the union $\cup_{P\in \cal P}\cal T_{\lambda, P}$. Note that $t\in\cal T_\lambda$ is an ideal triangle.

\begin{example}[Standard maximal geodesic lamination]\label{ex:prefmaxgeo}
Let us describe the \emph{standard maximal geodesic lamination $\kappa$}, which is a maximal geodesic lamination well-adapted to our purposes. Consider a pair of pants $P\in \cal P$, with boundary geodesics $g_1,g_2,g_3$ oriented so that $P$ lies to their left. For $i\in\bb Z_3$, let $e_i$ be a simple, bi-infinite geodesic spiraling around $g_{i-1}$, and $g_{i+1}$ in the direction \emph{opposite} to the orientation of $g_{i-1}$ and $g_{i+1}$, and oriented towards $g_{i+1}$. Set $\kappa_{P}=\{e_1,e_2,e_3\}$ and denote by $\kappa$ the resulting maximal geodesic lamination, which we sketch in Figure \ref{fig:stdgeod}.
\end{example}

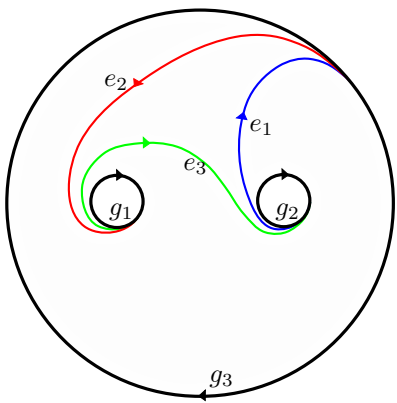
\begin{figure}[h]
\includegraphics[scale=.4]{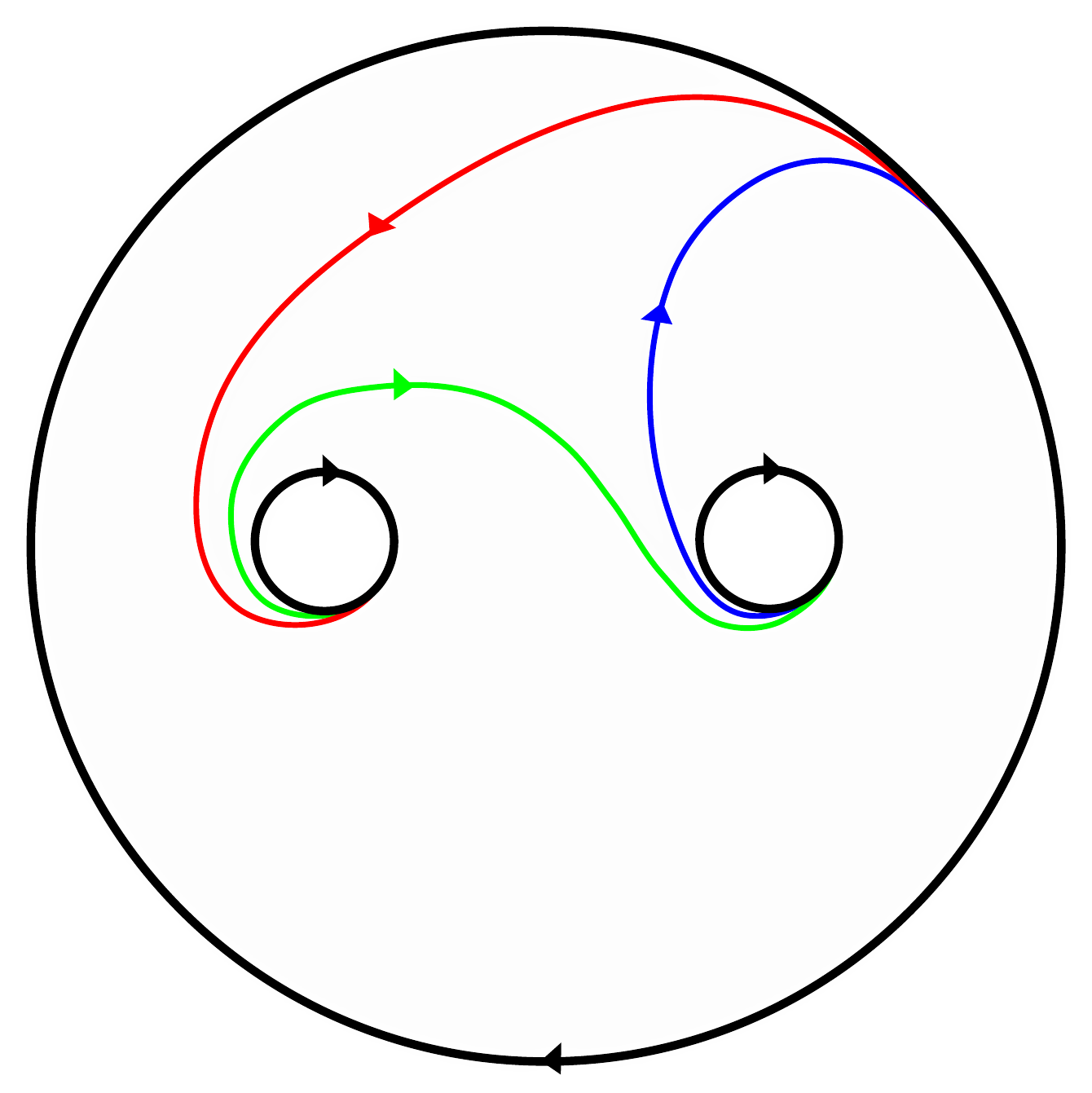}
\put (-395, 10){\makebox[0.7\textwidth][r]{\small{$g_3$}}}
\put (-370,73){\makebox[0.7\textwidth][r]{\small{$g_2$}}}
\put (-433,73){\makebox[0.7\textwidth][r]{\small{$g_1$}}}
\put (-405, 90){\makebox[0.7\textwidth][r]{\small{$e_3$}}}
\put (-380,105){\makebox[0.7\textwidth][r]{\small{$e_1$}}}
\put (-435,122){\makebox[0.7\textwidth][r]{\small{$e_2$}}}
\caption{The standard geodesic lamination of a pair of pants $P$ (shaded) described in Example \ref{ex:prefmaxgeo}.}
\label{fig:stdgeod}
\end{figure}

Given a maximal geodesic lamination $\lambda$, in order to parametrize the Hitchin component, one needs to further fix a collection $\cal K=\cal K(\cal P,\lambda)=\{k_1,\dots, k_{3g-3}\}$ of geodesic arcs $k_i$ with endpoints in $S-\lambda$ that intersect transversely and exactly once the closed leaf $c_i\in\lambda_c$. We will refer to the pair $(\lambda,\cal K)$ as a maximal geodesic lamination and drop $\cal K$ from the notation if clear from context.

Note that given the standard maximal geodesic lamination $\kappa_P$ of a pair pants $P$, we can replace $e_i\in \kappa_P$ with a bi-infinite geodesic with both ends which spiral around $g_i$. This gives a new maximal geodesic lamination $\lambda_P$ of $P$. We refer to this operation and its inverse
as \emph{(generalized) diagonal flips}. The analogy with \S \ref{ssec:flips} is clarified by lifting the picture to the universal cover.

\begin{notation} It will be useful to lift a maximal geodesic lamination $\lambda$ on $S$ to the universal cover $\wt S$. We will refer to lifts of $\lambda$, $\lambda_c$, $\cal T_\lambda$, $e$, etc. as $\wt\lambda$, $\wt\lambda_c$, $\wt{\cal T}_\lambda$, $\wt e$, etc., respectively.
\end{notation}

\subsubsection{Bonahon-Dreyer's paramatrization} We are now ready to describe the Bonahon-Dreyer coordinates of a Hitchin representation $\rho$ with respect to a maximal geodesic lamination $\lambda$ extending the pants decomposition $\cal P$. There are two types of parameters: the triangle parameters, and the shear parameters along a leaf.

Consider one of the pairs of pants $P\in\cal P$, and the corresponding sets $\lambda_{P}=\{e_1,e_2,e_3\}$, and $\cal T_{\lambda, P}=\{t,t'\}$. Choose two adjacent lifts $\wt t$ and $\wt t'$ sharing an oriented edge $\wt e$. This defines an ideal quadrilateral with diagonal $\wt e$. Following Notation \ref{not:vertices}, we denote the cyclically ordered vertices of this quadrilateral by $x_{e^+},x_{e^l},x_{e^-},x_{e^r}$. Up to relabeling, $\wt t$ has ideal vertices $x_{e^+},x_{e^l},x_{e^-}$, $\wt t'$ has ideal vertices $x_{e^+},x_{e^-},x_{e^r}$. 

Let $\rho\in\sf {Hit}_m(S)$ be a Hitchin representation, and recall that its limit map $\xi_\rho$ is \emph{positive} by Theorem \ref{thm:posflags}. For $a,b,c\in\bb Z_{>0}$ such that $a+b+c=m$, the \emph{$abc$-triangle parameter} of the triangle $t\in\cal T_{\lambda}$ is
 \[
\tau^{abc}(\rho,t):=\log T^{abc}(\xi_\rho(x_{e^+}),\xi_\rho(x_{e^l}),\xi_\rho(x_{e^-}))
\]
\begin{remark}\label{rmk:triangleinv} (Compare to Remark \ref{rmk:tripleratio}).
Note that the triangle parameters of an ideal triangle depend on a choice of a vertex. This ambiguity leads to the \emph{rotation conditions} when parametrizing the Hitchin component \cite[\S 5.1]{BoDr14}. On the other hand, the triangle parameters do not depend on the lift $\wt t$ of $t$ because the limit map is $\rho$-equivariant and the triple ratios are constant on the $\rm {PGL}_m(\bb R)$-orbit of the triple of flags $(\xi_\rho(x_{e^+}),\xi_\rho(x_{e^l}),\xi_\rho(x_{e^r}))$.
\end{remark}

With the same notations, for $a=1,\dots,m-1$, define the \emph{$a$-shearing} of $\rho$ along an open leaf $e\in\lambda_o$ as
\[
\sigma^a(\rho,e)=\log D^a(\xi_\rho(x_{e^+}),\xi_\rho(x_{e^l}),\xi_\rho(x_{e^-}),\xi_\rho(x_{e^r})).
\]
As in Remark \ref{rmk:triangleinv}, we observe that the shearings do not depend on the choice of lift $\wt e\in\wt \lambda_o$.

One uses the small transverse arcs in $\cal K$ to define the shearing parameters along a closed leaf. Consider a closed leaf $c$ and the corresponding transverse arc $k$. Choose a lift $\wt c$, and a lift $\wt k$ which intersects $\wt c$. The endpoints of $\wt k$ single out two ideal triangles $\wt t$, and $\wt t'$ which are lifts of ideal triangles in $\cal T_\lambda$. Note that $\wt k$ intersect exactly one edge $\wt e$ (resp. $\wt e'$) of $\wt t$ (resp. $\wt t'$). Let $(x_{c^+},x_{c^l},x_{c^-},x_{c^r})$ be the four points in $\bdry \wt S$ such that $x_{c^+}$ is the attracting fixed point of $\wt c$, $x_{c^l}$, $x_{c^r}$ are vertices of $\wt t$, $\wt t'$ that do not belong to $\wt e$ or $\wt e'$, $x_{c^-}$ is the repelling fixed point of $\wt c$. Up to relabeling, we have $(x_{c^+},x_{c^l},x_{c^-},x_{c^r})$ appear in this counter-clockwise cyclic order along $\bdry \wt S$.
For $a=1,\dots, m-1$, the \emph{$a$-shearing} of $\rho$ along the closed leaf $c$ is
\[
\sigma^a(\rho,c)= \log D^a(\xi_\rho(x_{c^+}),\xi_\rho(x_{c^l}),\xi_\rho(x_{c^-}),\xi_\rho(x_{c^r})).
\]
The $\rho$-equivariance of the limit map guarantees that $\sigma^a(\rho,c)$ does not depend on the lifts of $c$ and of the small transverse arc $k$.

\begin{thm}[Theorem 2 \cite{BoDr14}]\label{thm:bdmain} The triangle and shearing parameters define a homeomorphism, depending on the maximal geodesic lamination $\lambda$, between the Hitchin component $\sf {Hit}_m(S)$ and a convex polytope in $\bb R^N$ defined by finitely many inequalities and equalities.
\end{thm}

The equalities and inequalities in the statement of Theorem \ref{thm:bdmain} arise from the rotation conditions for the triangle parameters and from a careful analysis of the conjugacy classes $\rho(c)$ for $c\in\lambda_c$. 

Let us explain more. Note that $c\in\lambda_c$ defines a non-trivial conjugacy class $c$ in $\pi_1(S)$. For every Hitchin representation $\rho$ label by 
\[
\lambda_1(\rho,c)\geq\dots \geq \lambda_{m}(\rho,c)>0
\] 
 the moduli of the eigenvalues of $\rho(\gamma)\in\rm{PSL}_m(\bb R)$, with $\gamma$ any representative for the conjugacy class $c$. Note that $\lambda_a(\rho, c)$ only depends on $\rho$ and $c$. 

\begin{thm}[Theorem 1.5 \cite{Lab06}, Theorem 1.13 \cite{FG06}]\label{thm:lengths} For every Hitchin representation $\rho$ and for every non-trivial conjugacy class $c$ in $\pi_1(S)$, the $\lambda_a(\rho, c)$'s are distinct.
\end{thm}

\begin{prop}[Proposition 13 \cite{BoDr14}]\label{prop:lengthequ} Let $\rho$ be a Hitchin representation and let $\kappa$ be the standard maximal geodesic lamination defined in Example \ref{ex:prefmaxgeo}. Let $g$ be a boundary component of a pair of pants $P$ in $\cal P$ oriented so that $P$ lies to its left. Let $e\in\kappa_{P}$ be oriented towards $g$ and let $e'\in\kappa_P$ be oriented away from $g$. Finally, set $\cal T_{\kappa, P}=\{t,t'\}$. Then, for all $a=1,\dots, m-1$
\begin{equation}\label{eqn:lengths}
\log\frac{\lambda_a}{\lambda_{a+1}}(\rho,g)=\sigma^{a}(\rho,e)+\sigma^{m-a}(\rho,e')+\sum_{b+c=a}\left(\tau^{(m-a)bc}(\rho,t)+ \tau^{(m-a)bc}(\rho,t')\right).
\end{equation}
\end{prop}

Note that formulas similar to Equations \ref{eqn:lengths} hold for maximal geodesic laminations (with finitely many leaves) other than $\kappa$ \cite[Proposition 13]{BoDr14}.

Equations \ref{eqn:lengths} combined with Theorem \ref{thm:lengths} force certain expressions of the triangle and shearing invariants to be positive. 
Moreover, Equations \ref{eqn:lengths} give two expressions for
\[
\log\frac{\lambda_a}{\lambda_{a+1}}(\rho,g)=\log\frac{\lambda_{m-a}}{\lambda_{m-a+1}}(\rho,g^{-1})
\] 
in terms of a priori different triangle and shearing invariants. We refer to the constraints on the triangle and shearing parameters provided by Proposition \ref{prop:lengthequ} as the \emph{Bonahon-Dreyer length relations}.

\section{Sequences of tree-type}\label{sec:trees}

We now introduce our main new definition.

\begin{defin}[Tree-type]\label{def:treetype}A sequence of Hitchin representations $(\rho_n)$ is of \emph{tree-type with respect to the maximal geodesic lamination $\kappa$} defined in Example \ref{ex:prefmaxgeo} if there exists a sequence of strictly positive real numbers $(r_n)_n\subset \bb R$ converging to 0, and such that
\begin{itemize}
	\item[(A)] for every ideal triangle $t\in\cal T_\kappa$, and for every $a,b,c\in\bb Z_{>0}$ with $a+b+c=m$
	\[
	\lim_{n\to\infty} r_n \tau^{abc}(\rho_n,t)=0,
	\]
	\item[(B)] for every open leaf $e\in\kappa_o$, the limit 
	\[
	\lim_{n\to\infty} r_n\sigma^a(\rho_n, e)
	\] 
	is either
	\begin{itemize}
	\item[(B1)] non-negative for all $a=1,\dots, m-1$, or
	\item[(B2)] non-positive for all $a=1,\dots, m-1$.
	\end{itemize}
\end{itemize}
We refer to $r_n$ as a \emph{scaling sequence} for $\rho_n$.
\end{defin}

As mentioned in the introduction, our standing assumption is that the scaling sequence is non-trivial: at least one of the limits in Definition \ref{def:treetype} is different from zero.

\begin{example} Recall that the Hitchin component contains a copy of the space of Fuchsian representations. It is well known that, given any maximal geodesic lamination $\lambda$, a Fuchsian representation $\rho$ in the Hitchin component is such that:
\begin{enumerate}[$(i)$]
	\item $\tau^{abc}(\rho, t)=0$ for all triangles $t\in\cal T_\lambda$, and for all $a,b,c\in\bb Z_{>0}$ with $a+b+c=m$;
	\item $\sigma^a(\rho, e)=\sigma^b(\rho,e)$ for all open leaves $e\in\lambda_o$ and for all $a,b=1,\dots, m-1$. 
\end{enumerate}	
	Therefore, sequences of Fuchsian representations are (the motivating) examples of sequences of Hitchin representations of tree-type.
\end{example}

\begin{example}\label{ex:genus2} We now describe an explicit example of a sequence of Hitchin representations of tree-type which are not in the Fuchsian locus of the Hitchin component. Set $m=3$ and let $S$ be the surface of genus $2$ obtained by gluing two pairs of pants $P$ and $Q$ as in Figure \ref{fig:example}. We denote by $\kappa$ the standard maximal geodesic lamination associated to the pants decomposition $\cal P=\{P,Q\}$.

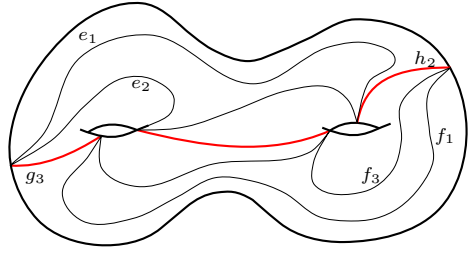
\begin{figure}[h]
\includegraphics[scale=.4]{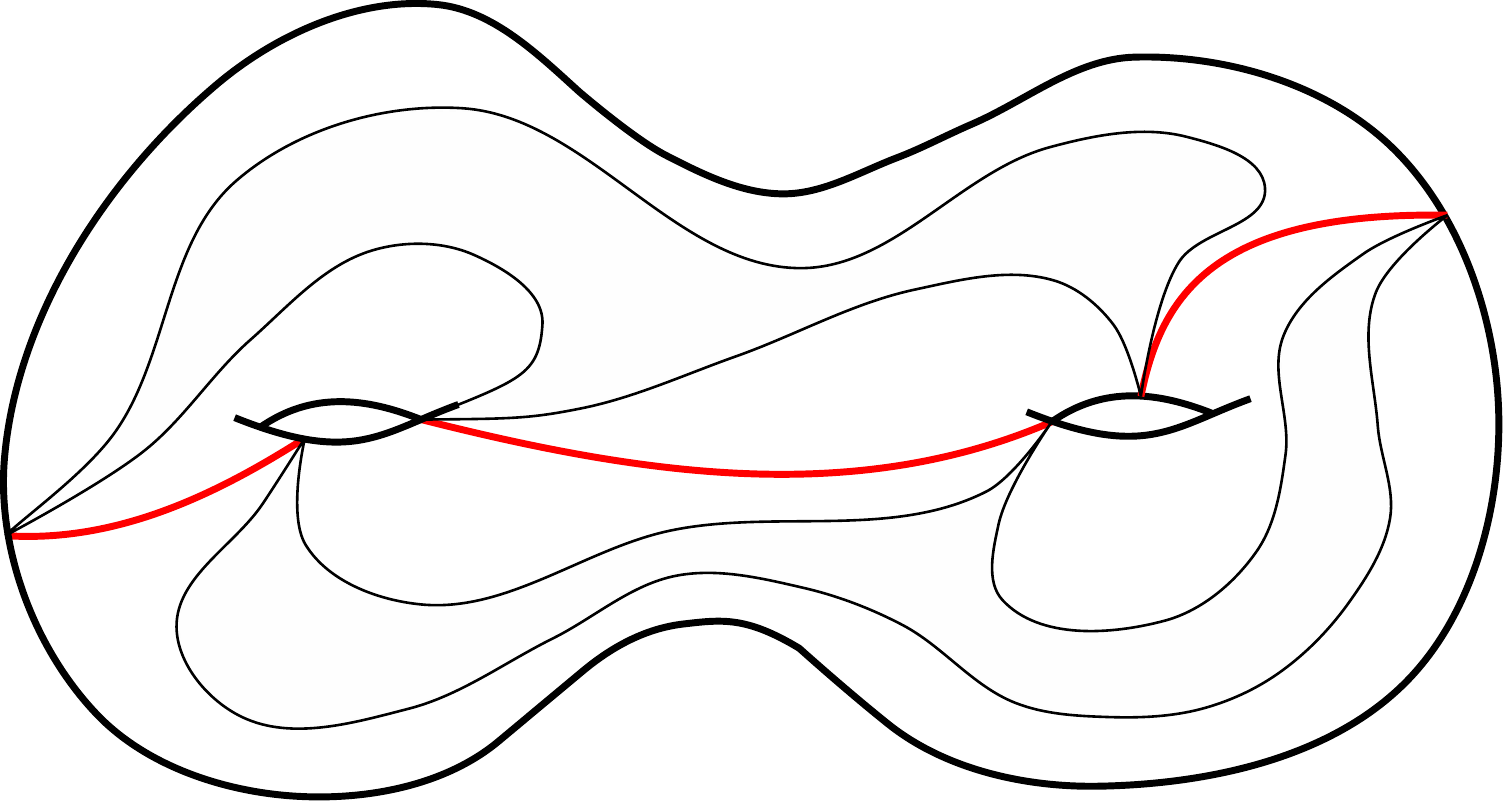}
\put (-470, 78){\makebox[0.7\textwidth][r]{\tiny{$e_1$}}}
\put (-450,60){\makebox[0.7\textwidth][r]{\tiny{$e_2$}}}
\put (-490, 25){\makebox[0.7\textwidth][r]{\tiny{$g_3$}}}
\put (-342, 69){\makebox[0.7\textwidth][r]{\tiny{$h_2$}}}
\put (-363,25){\makebox[0.7\textwidth][r]{\tiny{$f_3$}}}
\put (-335, 40){\makebox[0.7\textwidth][r]{\tiny{$f_1$}}}
\caption{The surface $S$ equipped with the maximal geodesic lamination $\kappa$ extending $\cal P$. The curves $g_i=h_i^{-1}$ are highlighted.}
\label{fig:example}
\end{figure}

Consider $P\in\cal P$ with boundary curves $g_1,g_2,g_3$ oriented so that $P$ lies to their left. Set $\kappa_{P}=\{e_1,e_2,e_3\}$, $\cal T_{\kappa,P}=\{t,t'\}$, and
\begin{align*}
\sigma^1(\rho_n, e_1)&=-n^2, & \sigma^1(\rho_n, e_2)&=n^2, & \sigma^1(\rho_n, e_3)&=2n, &  \tau^{111}(\rho_n, t)&=n\\
\sigma^2(\rho_n, e_1)&=-n, & \sigma^2(\rho_n, e_2)&=2n^2, & \sigma^2(\rho_n, e_3)&=2n^2, &  \tau^{111}(\rho_n,t')&=-n.
\end{align*}

Likewise, consider $Q\in\cal P$ with boundary curves $h_i=g_i^{-1}$ for $i=1,2,3$. Set $\kappa_{Q}=\{f_1,f_2,f_3\}$, $\cal T_{\kappa,Q}=\{u,u'\}$, and
\begin{align*}
\sigma^1(\rho_n, f_1)&=0, & \sigma^1(\rho_n, f_2)&=2n^2+n, & \sigma^1(\rho_n, f_3)&=2n^2+n, &  \tau^{111}(\rho_n, u)&=5n\\
\sigma^2(\rho_n, f_1)&=-n^2-n, & \sigma^2(\rho_n, f_2)&=n^2-n, & \sigma^2(\rho_n, f_3)&=n, &  \tau^{111}(\rho_n,u')&=-5n.
\end{align*}

Observe that these choices give a sequence of Hitchin representations as they satisfy the Bonahon-Dreyer length relations for $n>1$. Explicitly,
\begin{align*}
\log\frac{\lambda_1}{\lambda_2}(\rho_n,g_1)=\log\frac{\lambda_2}{\lambda_3}(\rho_n,h_1)&=2n+2n^2>0\\
\log\frac{\lambda_2}{\lambda_3}(\rho_n,g_1)=\log\frac{\lambda_1}{\lambda_2}(\rho_n,h_1)&=3n^2>0\\
 \log\frac{\lambda_1}{\lambda_2}(\rho_n,g_2)=\log\frac{\lambda_2}{\lambda_3}(\rho_n,h_2)&=n^2>0\\
 \log\frac{\lambda_2}{\lambda_3}(\rho_n,g_2)=\log\frac{\lambda_1}{\lambda_2}(\rho_n,h_2)&=n>0\\
  \log\frac{\lambda_1}{\lambda_2}(\rho_n,g_3)=\log\frac{\lambda_2}{\lambda_3}(\rho_n,h_3)&=n^2-n>0\\
 \log\frac{\lambda_2}{\lambda_3}(\rho_n,g_3)=\log\frac{\lambda_1}{\lambda_2}(\rho_n,h_3)&=n^2>0
\end{align*}

This construction defines a tree-type sequence of Hitchin representations for the surface $S$ of genus 2 with respect to the maximal lamination $\kappa$ with scaling sequence $r_n=1/n^2$. 
\end{example}

\begin{example}\label{exe:postriang3D} This example shows that, in general, property (A) in Definition \ref{def:treetype} is not preserved under diagonal flip of an open leaf of a maximal geodesic lamination $\lambda$.

Suppose $\rho$ is a Hitchin representation in $\sf{Hit}_3(S)$. Let $x_{e^+},x_{e^l},x_{e^-},x_{e^r}$ denote the vertices of two adjacent triangles $\wt t,\wt t'$ in $\wt{\cal T}_{\lambda}$ sharing the diagonal $e$ with endpoints $x_{e^+}$ and $x_{e^-}$. Let us simplify notations by setting
\begin{gather*}
X=T^{111}(\xi_\rho(x_{e^+}),\xi_\rho(x_{e^l}),\xi_\rho(x_{e^-})),\ Y= T^{111}(\xi_\rho(x_{e^+}),\xi_\rho(x_{e^-}),\xi_\rho(x_{e^r}))\\
Z=D^1(\xi_\rho(x_{e^+}),\xi_\rho(x_{e^l}),\xi_\rho(x_{e^-}),\xi_\rho(x_{e^r})),\ W=D^2(\xi_\rho(x_{e^+}),\xi_\rho(x_{e^l}),\xi_\rho(x_{e^-}),\xi_\rho(x_{e^r}))
\end{gather*}
An explicit computation shows that
\begin{gather*}
T^{111}(\xi_\rho(x_{e^l}),\xi_\rho(x_{e^-}),\xi_\rho(x_{e^r}))=X\frac{1+W+WY+ZWY}{1+Z+ZX+ZWX},\\ T^{111}(\xi_\rho(\xi_\rho(x_{e^l}),\xi_\rho(x_{e^r}),\xi_\rho(x_{e^+}))=Y\frac{1+Z+ZX+ZWX}{1+W+WY+ZWY}
\end{gather*}
It follows that one can find sequences of Hitchin representations that satisfy property (A) in Definition \ref{def:treetype} with respect to $\kappa$, but they do not satisfy the same property when considering the maximal geodesic lamination obtained from $\kappa$ by diagonal flip of one open leaf. For an explicit example, set $r_n=1/n$, $X_n=Y_n=1$, $W_n=e^n$ and $Z_n=e^{-n}$ and observe that in this case
\[
\lim_nr_n\log T^{111}(\xi_\rho(x_{e^l}),\xi_\rho(x_{e^-}),\xi_\rho(x_{e^r}))=1.
\]
\end{example}

\section{Preferred lamination on a pair of pants}\label{ssec:pairofpants}

This section is dedicated to the proof of Theorem \ref{thm:FLP}. Informally, this result states that item (B) makes Definition \ref{def:treetype} of tree-type sequences of Hitchin representations \emph{``stable under diagonal flip''}. More precisely, it fixes the issue highlighted in Example \ref{exe:postriang3D}.

\begin{thm}\label{thm:FLP} Let $(\rho_n)$ be a sequence of Hitchin representations of tree-type with respect to the maximal geodesic lamination $\kappa$ described in Example \ref{ex:prefmaxgeo}. Then, there exists a maximal geodesic lamination $\lambda^+=\lambda^+(\cal P,\rho_n)$ of $S$ extending the pants decomposition $\cal P$, and such that for every open leaf $e\in \lambda_o^+$ and triangle $t\in \cal T_{\lambda^+}$ we have
\begin{align}
\lim_{n\to\infty}r_n\sigma^a(\rho_n,e)\geq 0,\label{eqn:positiveshear}\\
\lim_{n\to\infty}r_n\tau^{abc}(\rho_n,t)=0.\label{eqn:positivetriangle}
\end{align}
\end{thm}

We subdivide the proof of Theorem \ref{thm:FLP} into several intermediate results to improve readability; the main step in the proof is Theorem \ref{thm:pmflips}.

\begin{lem}\label{lem:flip} Consider a sequence of Hitchin representations of tree-type $(\rho_n)$ with respect to the standard maximal geodesic lamination $\kappa$ defined in Example \ref{ex:prefmaxgeo} and a scaling sequence $r_n$. For each pair of pants $P\in\cal P$, there exists at most one open leaf $e\in\kappa_{P}$ with
\[
\lim_{n\to \infty} r_n\sigma^a(\rho_n,e)< 0.
\]
for some $a\in\{1,\dots, m-1\}$.
\end{lem} 
\begin{proof} This is a direct consequence of the Bonahon-Dreyer length relations applied to the boundary curves of $P$. Let us explain more. Assume, by contradiction, that there exist two leaves $e\neq e'\in\kappa_{P}$ such that
\[
\lim_{n\to \infty} r_n\sigma^a(\rho_n,e)< 0\text{, and }\lim_{n\to \infty} r_n\sigma^{a'}(\rho_n,e')< 0
\]
for some $a,a'\in\{1,\dots, m-1\}$. By definition of $\kappa$, there exists a boundary curve $g$ of $P$, oriented so that $P$ lies to the left of $g$, and such that $e$ and $e'$ spiral around $g$. Up to relabeling $e$ and $e'$, let us assume that $e$ (resp. $e'$) is oriented towards (resp. away from) $g$. Denote by $t$ and $t'$ the ideal triangles in $\cal T_{\kappa, P}$. Recall that we denote by $\lambda_a(\rho_n, g)$ the modulus of the $a$th largest eigenvalue of $\rho_n(g)$. Let $r_n>0$ be a scaling sequence given by Definition \ref{def:treetype}. Equation \ref{eqn:lengths} implies that for every $a=1,\dots, m-1$ and for every $n\in\bb N$
\begin{align*}
0<r_n\log\frac{\lambda_a}{\lambda_{a+1}}(\rho_n,g)&=r_n\sigma^a(\rho_n,e)+r_n\sigma^{m-a}(\rho_n,e')+r_n\sum_{b+c=a}\left( \tau^{(m-a)bc}(\rho_n,t)+\tau^{(m-a)bc}(\rho_n,t')\right).
\end{align*}
By Definition \ref{def:treetype}, for every $\varepsilon>0$, and for $n$ large $r_n\tau^{(m-a)bc}(\rho_n, t)$, $r_n\tau^{(m-a)bc}(\rho_n,t)<\varepsilon$. Let $a$ be an index such that  $\lim_{n\to \infty} r_n\sigma^a(\rho_n,e)< 0$. We assumed, by contradiction, that there exists $a'$ such that $\lim_n r_n\sigma^{a'}(\rho_n,e')< 0$. Definition \ref{def:treetype} implies that $\lim_n r_n\sigma^{m-a}(\rho_n,e')\leq 0$. Up to possibly choosing a smaller $\varepsilon$ and a larger $n$, we have
\[
r_n\sigma^{m-a}(\rho_n,e')<\varepsilon,\text{ and }r_n\sigma^a(\rho_n,e)<-4\varepsilon.
\]
Therefore, for $n$ large
\[
r_n\sigma^a(\rho_n,e')+r_n\sigma^{m-a}(\rho_n,e)+r_n\sum_{a+b+c=m}\left( \tau^{(m-a)bc}(\rho_n,t)+\tau^{(m-a)bc}(\rho_n,t')\right)<-4\varepsilon +\varepsilon+\varepsilon+\varepsilon<0
\]
which contradicts the Bonahon-Dreyer length relations.
\end{proof}

In other words, Lemma \ref{lem:flip} states that on each pair of pants $P$, item (B2) in Definition \ref{def:treetype} holds for at most one open leaf in $\kappa_{P}$. 

\begin{notation}\label{not:poslam} The maximal geodesic lamination $\lambda^{+}=\lambda^+(\cal P,\rho_n)$ is obtained by performing a diagonal flip to all the open leaves $e$ in the standard maximal geodesic lamination $\kappa$ such there exists $a=1,\dots, m-1$ with
\[
\lim_{n\to \infty} r_n\sigma^a(\rho_n,e)< 0.
\]
\end{notation} 

In order to prove that Equations \ref{eqn:positiveshear}-\ref{eqn:positivetriangle} hold for $\lambda^+$ for all $m\geq 2$, we will rely on the finer structure of the Fock-Goncharov parameters summarized in \S \ref{ssec:flips}. 

Let $e\in\kappa_{P}$ be \emph{any} open leaf, and consider a lift $\wt e$ to the universal cover $\wt S$. This lift is an oriented diagonal for an ideal quadrilateral in $\wt S$ composed by two adjacent ideal triangles that are lifts of triangles in $\cal T_{\lambda,P}$. Denote by $(x_{e^+},x_{e^l},x_{e^-},x_{e^r})$ the cyclically ordered vertices of this quadrilateral, with $x_{e^+}$ the attracting fixed point of $\wt e$.

By Theorem \ref{thm:poslimmap}, for every $n\in \bb N$, the quadruple of flags $(\xi_n(x_{e^+}),\xi_n(x_{e^l}),\xi_n(x_{e^-}),\xi_n(x_{e^r}))$ is positive. This sequence of positive quadruples of flags corresponds to sequences of weighted directed graphs $\cal Q_{r}(e,n)$, $r=0,\dots, r_{\max}$ as described in \S\ref{ssec:flips}. Our next step is to replace the sequences $\cal Q_{r}(e,n)$ of weighted directed graphs with simplified versions. 

Explicitly, let us start by defining the sequence of weighted directed graphs $\wt {\cal Q}_0(e, n)$ with: 
\begin{itemize}
	\item[-] vertices and arrows equal to vertices and arrows of $\cal Q_{0}(e,n)$;
	\item[-] weights on the interior vertices $v\in\rm{int}(\wt {\cal Q}_{0}(e,n))$ given by
	\[
	\wt X_n^v=
	\begin{cases}
	e^{\sigma^a(\rho_n,e)} &\text{ for }v=(a,b,0,0), \ a>0\\
	1 &\text{ otherwise}
	\end{cases}
	\]
\end{itemize}
The weighted directed graph $\wt {\cal Q}_0(e, n)$ is an asymptotic version of $\cal Q_0(e,n)$ in the sense that for every interior vertex $v\in\rm{int}(\wt {\cal Q}_{0}(e,n))$
\begin{equation}\label{eqn:asymptotic}
\lim_nr_n\log X_n^v=\lim_nr_n\log \wt X_n^v,
\end{equation}
where we denote by $X_n^v$ the sequence of weights of a vertex $v$ in $\cal Q_{0}(e,n)$.
Equation \ref{eqn:asymptotic} holds even after a mutation thanks to the following series of observations.
\begin{lem}\label{lem:limits1} Let $(r_n)$, and $(A_n)$ be sequences such that 
\[
r_n, A_n>0,  \lim_n r_n =0, \text{and }\lim_n r_n\log A_n\leq 0.
\]
Then, $\lim_n r_n \log (1+A_n)=0.$
\end{lem}
\begin{proof} Note that $r_n\log(1+A_n)>0$ for all $n$, so $\lim_n r_n \log (1+A_n)\geq 0$. Set $K=\lim_nr_n\log A_n$. For every $\varepsilon>0$ we can find $n$ large enough so that $A_n<e^{(K+\varepsilon)/r_n}$, which implies
\[
\lim_n r_n \log (1+A_n)\leq \lim_n r_n\log\left(1+e^{(K+\varepsilon)/r_n}\right).
\]
Observe that
\begin{equation}\label{eqn:useful}
\lim_{x\to 0}x\log(1+C^{1/x})=
\begin{cases}
0&\text{ if } 0<C\leq 1\\
\log C &\text{ if }C>1
\end{cases}
\end{equation}
The result follows by setting $C=e^{K+\varepsilon}$, since $\veps$ was arbitrary.
\end{proof}
\begin{cor}\label{cor:limits2} Let $(r_n)$, and $(B_n)$ be sequences such that 
\[
r_n, B_n>0,\ \lim_n r_n=0,\text{and }\lim_n r_n\log B_n\geq 0.
\]
Then, $\lim_n r_n \log (1+B_n)=\lim_n r_n\log B_n.$
\end{cor}
\begin{proof} Apply Lemma \ref{lem:limits1} with $A_n=1/B_n$. 
\end{proof}

\begin{prop}\label{prop:asymptotics} Let $v$ be a vertex in the interior of $\cal Q_0$. Denote by $X_n^w$ and $\wt X_n^w$ the weights of a vertex $w$ for the weighted directed graphs $\cal Q_0(e,n)$ and $\wt {\cal Q}_0(e,n)$, respectively. Then
\[
\lim_n r_n\log \mu_v(X_n^w)=\lim_n r_n\log \mu_v(\wt X_n^w).
\]
\end{prop}
\begin{proof} By definition, a mutation only changes the weights of five vertices in $\cal Q_0$. See Figure \ref{fig:mutation}. There are a few cases to consider.

\textbf{Case 1:} Assume $\lim_n r_n \log X^v_n=\lim_n r_n \log \wt X^v_n\leq 0$ and $\mu_v(X_n^w)=X_n^w(1+X^v_n)$. Then,
\begin{align*}
\lim_nr_n\log \mu_v(X_n^w)&=\lim_nr_n\log X_n^w(1+X^v_n)\\
&=\lim_nr_n\log X_n^w + \lim_nr_n\log (1+X^v_n)\\
&=\lim_nr_n\log X_n^w\\
&=\lim_nr_n\log \wt X_n^w\\
&=\lim_nr_n\log \wt X_n^w+\lim_nr_n\log (1+\wt X^v_n)=\lim_nr_n\log \wt \mu_v(\wt X_n^w)
\end{align*}
where we used Lemma \ref{lem:limits1} on the third and fifth line.

\textbf{Case 2:} Assume  $\lim_n r_n \log X^v_n=\lim_n r_n \log \wt X^v_n\geq 0$ and $\mu_v(X_n^w)=X_n^w(1+X^v_n)$. Then,
\begin{align*}
\lim_nr_n\log \mu_v(X_n^w)&=\lim_nr_n\log X_n^w(1+X^v_n)\\
&=\lim_nr_n\log X_n^w + \lim_nr_n\log (1+X^v_n)\\
&=\lim_nr_n\log X_n^w+\lim_nr_n\log X^v_n\\
&=\lim_nr_n\log \wt X_n^w+\lim_nr_n\log \wt X^v_n\\
&=\lim_nr_n\log \wt X_n^w+\lim_nr_n\log (1+\wt X^v_n)=\lim_nr_n\log\mu_v(\wt X_n^w)
\end{align*}
where we used Corollary \ref{cor:limits2} on the third and fifth line.

The remaining cases $w=v$ and $\mu_v(X_n^w)=X_n^w\frac{X_n^v}{1+X^v_n}$ can be checked similarly. We omit the details.
\end{proof}

Proposition \ref{prop:asymptotics} motivates the definition of \emph{asymptotic mutation at a vertex $v$:}
\[
\wt \mu_v(X_n^w)= 
\begin{cases}
1/X_n^v&\text{ if }v=w\\
X_n^wX_n^v &\text{ if $v$ is connected to $w$, and }\lim_nr_n\log\mu_v(X_n^w)=\lim_nr_n\log X_n^wX_n^v\\
X_n^w &\text{ otherwise}
\end{cases}
\]
With these notations, Proposition \ref{prop:asymptotics} implies
\begin{equation}\label{eqn:asmutation}
\lim_nr_n\log(\mu_v(X_n^w))=\lim_nr_n\log(\wt \mu_v(\wt X_n^w)).
\end{equation}

\begin{thm}\label{thm:pmflips} Let $(\rho_n)$ be a sequence of Hitchin representations of tree-type with respect to the maximal geodesic lamination $\kappa$ described in Example \ref{ex:prefmaxgeo}. Let $\cal Q_0(e,n)$ be the weighted directed graph associated to $\rho_n$ by choosing an open leaf $e$ in $\kappa_o$. Let $\lambda$ be the maximal geodesic lamination obtained from $\kappa$ by diagonal flip of the open leaf $e$. Then,
\[
\lim_nr_n\tau^{abc}(\rho_n,t)=0,\text{ for all the ideal triangles }t\in\cal T_{\lambda}.
\]
\end{thm}
\begin{proof} We start by easing notation and writing $\tau^{abc}_n(\cdot)$ and $\sigma^a_n(\cdot)$ instead of $\tau^{abc}(\rho_n,\cdot)$ and $\sigma^a(\rho_n,\cdot)$. 

Let us assume that for every $a=1,\dots, m-1$,
\[
\lim_nr_n\sigma_n^a(e)\leq 0.
\]
We omit the details of the proof of the case $\lim_nr_n\sigma_n^a(e)\geq 0$ as it is analogous. Furthermore, let us first focus on the vertices of the form $(a,b,c,0)$ in the discrete parallelogram $\cal Q_0$. 

We are interested in understanding the asymptotic behavior of $\nu_0(X^v_n)$, where $\nu_{0}=\bar \mu_{m-2}\circ\dots\circ\bar \mu_0$ is the composition of strata-mutations described in \S \ref{ssec:flips}, and $X_n^v$ is the weight of $v$ in $\cal Q_0(e,n)$.

Set $v_{\beta,c}=(\alpha-c,\beta,c,0)$ for $\beta=1,\dots, m-1$, $\alpha=m-\beta$ and $c=0,\dots, \alpha$. We claim that, with our assumption on the asymptotic behavior of the shearings on $e$, we can replace $\nu_0$ with mutations
\[
\wt \mu_{v_{\beta,\alpha}}\circ\dots\circ \wt \mu_{v_{\beta,0}}
\]
as we let $\beta$ vary between $1$ and $m-1$. To explain why this is the case, we start by observing that if the weight $X^v_n$ of the vertex $v=(a-1,b+1,c,0)$ is such that 
\[
\lim_nr_n\log X^v_n\leq 0,
\]
then the weight $X^w_n$ of the vertex $w=(a-1,b,c+1,0)$ is such that $\wt \mu_v(X^w_n)=X^w_n$. In this case, the result of a $c$-stratum mutation on the weight $X^w_n$ is the same as the result of a mutation at the vertex $(a,b,c,0)$.

Fix $\beta=1,\dots, m-1$. Assume that, after performing asymptotic mutations $\wt \mu_{v_{\beta,c}}\circ\dots\circ \wt \mu_{v_{\beta,0}}$, the weights $\wt Y^\gamma_n$ of the vertices $v_{\beta,\gamma}$ are such that 
\[
\wt Y^\gamma_n=
\begin{cases}
e^{-\sigma^{\alpha}_n(e)}&\text{ for }\gamma=c-1\\
e^{\sigma^{\alpha}_n(e)}&\text{ for }\gamma=c\\
1&\text{ for }\gamma=c+1
\end{cases}
\]
Note that this is the case, for example, when $c=1$. Then, by definition of asymptotic mutation
\begin{equation}\label{eqn:asvertexmut}
\wt \mu_{v_{\beta,c}}(\wt Y^\gamma_n)=
\begin{cases}
1&\text{ for }\gamma=c-1\\
e^{-\sigma^{\alpha}_n(e)}&\text{ for }\gamma=c\\
e^{\sigma^{\alpha}_n(e)}&\text{ for }\gamma=c+1
\end{cases}
\end{equation}

Applying Equations \ref{eqn:asvertexmut} iteratively we get that for every $\beta\in \{1,\dots, m-1\}$, if $\wt X^{v_{\beta,c}}_n$ denotes the weight of the vertex $v_{\beta,c}$ in $\wt {\cal Q}_0(e,n)$, then
\[
(\wt \mu_{v_{\beta,\alpha}}\circ\dots\circ\wt\mu_{v_{\beta,0}})(\wt X^{v_{\beta,c}}_n)=
\begin{cases}
1&\text{ for } 0<c<\alpha-1\\
e^{-\sigma^{\alpha}_n(e)}&\text{ for }c=\alpha-1\\
e^{\sigma^{\alpha}_n(e)}&\text{ for }c=\alpha
\end{cases}
\]
The same reasoning applied to the vertices of the form $(a,b,0,d)$ in $\cal Q_0$ shows that for every $\alpha\in\{1,\dots, m-1\}$ the weights of $\wt{\cal Q}_0(e,n)$ after applying $\nu_0$ behave asymptotically as
\[
\begin{cases}
1&\text{ for }v=(\alpha, m-\alpha-d,0, d),\ 0<d<m-\alpha-1\\
e^{-\sigma^{\alpha}_n(e)}&\text{ for }v=(\alpha,1,0,m-\alpha-1)\\
e^{\sigma^{\alpha}_n(e)}&\text{ for }v=(\alpha, 0,0, m-\alpha)
\end{cases}
\]
It remains to describe the asymptotic behavior of the weight of a vertex of the form $v=(a,b,0,0)$ on the preferred diagonal of $\wt {\cal Q}_0(e,n)$ after applying $\nu_0$. 

When $a$ or $b$ is equal to 1, the only mutations we need to check are $\wt\mu_w$ with $w=(a,b-1,0,1)$ and $\wt \mu_v$. An easy computation shows
\[
(\wt\mu_{w}\circ\wt\mu_v)(\wt X^v_n)=1.
\] 

Assume $a,b>1$ and set $v'=(a-1,b,1,0)$ and $v''=(a,b-1,0,1)$. Then,
\begin{align*}
\wt X^v_n=e^{\sigma^a_n(e)}&\stackrel{\wt \mu_v}{\longmapsto} e^{-\sigma^a_n(e)}\stackrel{\wt \mu_{v'}}{\longmapsto} e^{-\sigma^a_n(e)} \cdot e^{\sigma^a_n(e)}=1\stackrel{\wt \mu_{v''}}{\longmapsto} e^{\sigma^a_n(e)}\cdot 1=e^{\sigma^a_n(e)}.
\end{align*}

We can iteratively define weighted directed graphs $\wt {\cal Q}_{r}(e,n)$ for $r=1,\dots, r_{\max}$ with $r_{\max}=\lfloor \frac{m}{2}\rfloor$ such that
\begin{itemize}
	\item[-] the vertices and arrows of $\wt {\cal Q}_{r}(e,n)$ are the vertices and arrows of ${\cal Q}_{r}(e,n)$
	\item[-] the weights of $\wt {\cal Q}_{r}(e,n)$ are given by replacing the result of $\nu_{r-1}$ applied to $\wt {\cal Q}_{r-1}(e,n)$ with their asymptotic counterpart. More precisely, if $\wt X^v_n$ is the weight of $v$ in $\wt{\cal Q}_{r-1}(e,n)$ we set the weight of $v$ in $\wt {\cal Q}_{r}(e,n)$ to be
	\[
	\wt Y_n^v=
	\begin{cases}
	1&\text{ if }\lim_nr_n\log\nu_{r-1}(\wt X^v_n)=0\\
	e^{\sigma^a_n(e)}&\text{ if }\lim_nr_n\log\nu_{r-1}(\wt X^v_n)=\lim_nr_n\sigma^a_n(e)\\
	e^{-\sigma^a_n(e)}&\text{ if }\lim_nr_n\log\nu_{r-1}(\wt X^v_n)=-\lim_nr_n\sigma^a_n(e)\\
	\end{cases}
	\]
\end{itemize}
The same argument described in the case of $\wt{\cal Q}_0(e,n)$ shows that $\wt {\cal Q}_{r}(e,n)$ has weights 
\[
\begin{cases}
e^{\sigma^a_n(e)}&\text{ for }v=(a,b,0,0),\ a>r\\
e^{-\sigma^{m-b-r}_n(e)}&\text{ for }v=(r, b,m-b-r,0),\\
e^{-\sigma^a_n(e)}&\text{ for }v=(a, r,0,m-a-r)\\
1&\text{ otherwise}
\end{cases}
\]
Applying the above formulas for all the $r=0,\dots, r_{\max}$, we obtain that every vertex $v=(a,b,c,d)$ in $\rm{int}(\cal Q_0)$ with $a\neq b$ has weight $1$ after performing the asymptotic version of the transformation $\nu_{r_{\max}-1}\circ\dots\circ\nu_0$. The result then follows by Lemma \ref{lem:strata} and Proposition \ref{prop:asymptotics}.
\end{proof}

\begin{proof}[Proof of Theorem \ref{thm:FLP}]

Let $\lambda^+$ be the maximal geodesic lamination defined in Notation \ref{not:poslam}.

We wish to prove that Equations \ref{eqn:positiveshear} and  \ref{eqn:positivetriangle} hold for $\lambda^+$. Equation \ref{eqn:positivetriangle} holds for $\lambda^+$ because of Theorem \ref{thm:pmflips}. Equation \ref{eqn:positiveshear} follows from the Bonahon-Dreyer length relations as follows. 

Suppose there exists a pair of pants $P\in\cal P$, and open leaves $e,e'\in\kappa_{P}$ such that
\[
\lim_nr_n\sigma_n^a(e)\leq 0,\text{and }\lim_nr_n\sigma_n^a(e')\geq 0.
\]
It follows from the proof of Theorem \ref{thm:pmflips} and from Remark \ref{rmk:generalflip} that the double ratios $\bar{\sigma}^{a}_n(e')$ of $e'$ with respect to $\lambda^+$ are such that
\[
\lim_{n} r_n\bar{\sigma}^{a}_n(e')=\lim_n r_n(\sigma^a_n(e')+\sigma^{m-a}_n(e))
\]
Let $t,t'\in\cal T_{\kappa,P}$ and let $g$ be the boundary geodesic oriented so that $P$ lies to its left and such that $e'$ and $e$ spiral around $g$. By the Equalities \ref{eqn:lengths} and Definition \ref{def:treetype}
\begin{align*}
\lim_nr_n \bar{\sigma}^{a}_n(e')&=\lim_nr_n \left(\sigma_n^a(e)+\sigma_n^{m-a}(e')+\sum_{b+c=a}\left(\tau^{(m-a)bc}(t)+\tau^{(m-a)bc}(t')\right)\right)\\
&=\lim_nr_n\log\frac{\lambda_b}{\lambda_{b+1}}(\rho_n,g)\geq 0
\end{align*}
where $b=a$ or $m-a$. This shows that $\lim_nr_n\bar{\sigma}^a(e')\geq 0$, as desired.
\end{proof}

\section{Asymptotics of Hilbert length}\label{ssec:hilbert}

We now wish to apply the results of \S \ref{ssec:pairofpants} to describe the asymptotic behavior of the \emph{Hilbert length} 
\[
\ell_H(\rho_n,c)=\log\frac{\lambda_1}{\lambda_m}(\rho_n, c)
\] 
where $(\rho_n)$ is a sequence of Hitchin representations of tree-type and $c$ is a conjugacy class in $\pi_1(S)$.

\begin{thm}\label{thm:action} Let $(\rho_n)$ be a sequence of Hitchin presentations of tree-type and let $\lambda^+$ be the maximal geodesic lamination given by Theorem \ref{thm:FLP}. Let $c$ be a conjugacy class in $\pi_1(S)$ such that
\[
\lim_nr_n\ell_H(\rho_n,c)\in [0,\infty).
\]
and either of the following assumptions holds.
\begin{enumerate}[a.]
	\item\label{cond:anyd} The curve $c$ has zero intersection number with the closed leaves in $\lambda^+_c$.
	\item\label{cond:d3} The sequence $(\rho_n)$ is contained in $\sf{Hit}_3(S)$, $\lim_nr_n\sigma^a(\rho_n,d)\geq 0$ for all $a=1,2$, and $d\in\lambda^+_c$, and there exists $C>0$ such that for all $a=1,2$, and $d\in\lambda^+_c$
	\[
	\log\frac{\lambda_a}{\lambda_{a+1}}(\rho_n,d)>C
	\] 
\end{enumerate} 
Then, there exists a finite set $\cal E$ of leaves of $\wt \lambda^+$ such that
\[
\lim_nr_n\ell_H(\rho_n,c)=\lim_nr_n \sum_{e\in\cal E}\sigma(\rho_n,e),
\]
where $\sigma(\rho_n,e)=\sigma^1(\rho_n,e)+\dots+\sigma^{m-1}(\rho_n,e)$.
\end{thm}
\begin{remark}
The finite set of leaves $\cal E$ in the statement of Theorem \ref{thm:action} depends on several choices that we describe in detail in \S \ref{ssec:finipoly2}. 
\end{remark}
\begin{example} Consider the sequence of Hitchin representations described in Example \ref{ex:genus2}. Observe that we can prescribe the shear parameters of the closed leaves of $\kappa_c=\lambda^+_c$ so that $\rho_n$ satisfies condition \ref{cond:d3} in Theorem \ref{thm:action}. 
\end{example}
The rest of this section is dedicated to the proof of Theorem \ref{thm:action}. In \S\S \ref{ssec:finitepoly}-\ref{ssec:finipoly2}, following \cite{FG06}, we explicitly build a \emph{totally positive} representative for the conjugacy class $\rho_n(c)$. The main technical result needed for our proof of Theorem \ref{thm:action}.\ref{cond:d3} is proved in \S\ref{ssec:dim3}. Finally, in \S \ref{ssec:proofaction} we establish Theorem \ref{thm:actionsuper}, which implies Theorem \ref{thm:action}.

\subsection{Total positivity and disconnecting triangulations}\label{ssec:finitepoly}
Let $\lambda$ be a triangulation of the polygon $\Pi$ inscribed in $S^1$ with cyclically ordered vertices $(x_1,\dots, x_k)$. Let $(F_1,\dots, F_k)\in \sf{F}^{(k+)}_m$ be a corresponding positive $k$-tuple of flags. Consider two triangles $t,u\in\cal T_{\lambda}$ corresponding to the positive triples of flags $(F_1,F_2,F_3)$ and $(G_1,G_2,G_3)$.
\begin{prop}[\S 9 \cite{FG06}]\label{prop:gtilde} With the notations above, there exists a unique $ g\in\rm{PGL}_m(\bb R)$ such that $g(F_i)= G_i$ if and only if
\[
T^{abc}(F_1,F_2,F_3)=T^{abc}(G_1,G_2,G_3),\ \text{for all }a,b,c\in\bb Z_{> 0},\ a+b+c=m.
\]
\end{prop}

Our goal is to use the triangulation $\lambda$ to describe a \emph{totally positive} representative $A_{g}\in\rm{GL}_m(\bb R)$ of $g$, i.e. such that the image of $A_{g}$ under the projection $\rm{GL}_m(\bb R)\to\rm{PGL}_m(\bb R)$ is $g$, and
\begin{itemize}
	\item[-] if $F_1\neq G_1$ and $F_3\neq G_3$, then all the minors of $A_{g}$ are positive;
	\item[-] otherwise, $A_{g}$ is a triangular matrix with positive minors, except the ones that are necessarily zero due to the shape of the matrix.
\end{itemize}
The matrix $A_{g}$ is described in terms of the Fock-Goncharov parameters $\phi_\lambda(F_1,\dots, F_k)$.

Up to considering a smaller polygon, we can assume that every vertex of the dual graph of $\lambda$ has valence two. In this case we say that $\lambda$ is a \emph{disconnecting} triangulation of $\Pi$. For example, the triangulation depicted in Figure \ref{fig:flip} is disconnecting. With this assumption, we can order the diagonals $\lambda=(e_1,\dots,e_{k-3})$, and the triangles $\cal T_\lambda=(t_0, \dots ,t_{k-3})$ in such a way that $t_{i-1}$ and $t_i$ share the edge $e_i$ for $i=1,\dots, k-3$. Observe that, for every $i=1,\dots k-4$, we can specify a preferred vertex $x_{t_i}$ of the triangle $t_i$ as the common endpoint of $e_i$ and $e_{i+1}$.  Let us set $x_{t_0}=x_1$, and $x_{t_{k-3}}$ equal to the vertex corresponding to the flag $G_1$.

Note that, by genericity of the flags, there exists a basis $(\vec u_1,\dots,\vec u_m)$ of $\bb R^m$ such that for all $a=1,\dots, m-1$
\begin{gather*}
F_1^{(a)}=\rm {Span}(\vec u_1,\dots, \vec u_a),\ \  F_3^{(a)}=\rm {Span}(\vec u_m,\dots, \vec u_{m-a+1}),\\
F_2^{(1)}=\rm {Span}(\vec u_1-\vec u_2+\dots + (-1)^{m-1}\vec u_m).
\end{gather*} 

In what follows, we will write matrices with respect to the basis $(\vec u_1,\dots,\vec u_m)$ using the coordinates described in \S \ref{ssec:coords}. Recall that every oriented diagonal $e_i\in\lambda$ corresponds to $m-1$ double ratios that we denote by $D^1(e_i),\dots, D^{m-1}(e_i)$. The \emph{edge matrix} $\frak D_i$ of $e_i$ is the diagonal matrix
\[
\frak D_i=\begin{cases}
\rm{diag}(1,D^{m-1}(e_i),\dots, D^{1}(e_i)\cdots D^{m-1}(e_i))& \text{ if $e_1$ is to the left of $e_i\in\lambda$}\\
\rm{diag}(1,D^{1}(e_i),\dots, D^{1}(e_i)\cdots D^{m-1}(e_i)) & \text{ otherwise} 
\end{cases}
\]
with the convention that $e_1$ lies to the left of itself.

Fix a triangle $t_i\in\cal T_\lambda$. We wish to define a triangular matrix, denoted by $\frak T_i$, which depends exclusively on the triple ratios $T^{abc}(t_i)$ computed with respect to the preferred vertex $x_{t_i}$ of the triangle $t_{i}$. We describe $\frak T_i$ as a product of diagonal and unipotent matrices.

Consider the diagonal matrix $\frak H^u_a(y)=\rm {diag}(y \cdot \rm {Id}_a, \rm {Id}_{m-a})$, where $\rm {Id}_b$ is the $b\times b$ identity matrix. Let $\frak U_a=I_m+E_{a,a+1}$, where $E_{a,a+1}$ is the matrix with 1 at the $(a,a+1)$-st entry and zeros elsewhere.

For $c=1,\dots, m-1$, consider the matrices
\[
\frak S^u_c(t_i)=\frak U_1\prod_{a=1}^{m-c-1}\frak H^u_{a+1}\left(\frac{1}{T^{abc}(t_i)}\right)
\frak U_{a+1},
\]
where $a+b+c=m$, $a,b\in\bb Z_{>0}$.

Set $\frak L_a$ to be the transpose of $\frak U_a$, and $\frak H^l_a(y)=\rm {diag}(\rm {Id}_a, y\cdot \rm {Id}_{m-a})$. For $c=1,\dots, m-1$ consider
\[
\frak S^{l}_c(t_i)=\frak L_{m-1}\prod_{b=1}^{m-c-1}\frak H^l_{m-(b+1)}(T^{abc}(t_i))\frak L_{m-(b+1)}
\]
where $a+b+c=m$, $a,b\in\bb Z_{>0}$. Set 
\[
\frak T_i=
\begin{cases}
\frak S^u_1(t_i)\dots \frak S^u_{m-1}(t_i)& \text{ if } (x_1,x_{t_i},y_1)\text{  appear in this cyclic order}\\
\frak S^l_{1}(t_i)\dots \frak S^l_{m-1}(t_i)&\text{ otherwise}
\end{cases}
\]
with the convention that $\frak T_{k-3}=\frak S^u_1(t_{k-3})\dots \frak S^u_{m-1}(t_{k-3})$.
\begin{prop}[Propositon 9.2 \cite{FG06}]\label{prop:multi} Let $(F_1,\dots, F_k)\in\sf{F}^{(k+)}_m$ be a positive $k$-tuple of flags. Consider $G_1,G_2,G_3\in\{F_1,\dots, F_k\}$ and assume  
\[
T^{abc}(F_1,F_2,F_3)=T^{abc}(G_1,G_2,G_3),\ \text{for all }a,b,c\in\bb Z_{> 0},\text{with }a+b+c=m.
\]
For $\lambda$, $\frak D_i$ and $\frak T_i$ defined above, the totally positive matrix
\[
A_{g}=\frak D_1\frak T_1\dots \frak D_{k-3}\frak T_{k-3}
\]
is a representative in $\rm{GL}_m(\bb R)$ of $g\in\rm{PGL}_m(\bb R)$ with $g(F_i)=G_i$.
\end{prop}

\subsection{Finite polygons associated to a fundamental group element}\label{ssec:finipoly2}

Consider a maximal geodesic lamination $\lambda$. Let $\gamma$ be a non-trivial group element in $\pi_1(S)$ with axis $\rm{ax}(\gamma)$ in $\wt S$. Choose a point $p\in \rm{ax}(\gamma)$ in the interior of a triangle $t_0\in\wt {\cal T}_{\lambda}$. 

In this section, we outline how to define a polygon $\Pi_{\gamma,p}$ with finitely many vertices inscribed in $\bdry\wt S$, and equipped with a disconnecting triangulation $\cal E=\cal E_{\gamma,p}$ in the sense of \S\ref{ssec:finitepoly}.

Label the vertices of $t_0$ by $x_1,x_2,x_3$ so that the leaf $\wt e\in\wt \lambda$ connecting $x_1$ and $x_3$ is the closest, amongst the edges of $t_0$, to the $\gamma$-translate $\gamma(t_0)$. Symmetrically, denote by $y_1$, $y_2$, $y_3$ the vertices of the edge of $\gamma(t_0)$ so that $\gamma(x_i)=y_i$, $i=1,2,3$. 

Observe that there exist finitely many lifts of closed leaves $\eta\in\wt\lambda_c$ such that $t_0$ and $\gamma(t_0)$ belong to different connected components of $\wt S-\eta$. For each such $\eta$, choose a lift $\wt k$ of an arc $\cal K$ such that $\wt k\cap \eta\neq \emptyset$. There exist finitely many lifts $e$ of open leaves such that 
\begin{enumerate}
	\item the leaf $e$ disconnects $t_0$ from $\gamma(t_0)$,
	\item the leaf $e$ does not intersect any of the chosen lifts $\wt k$.
\end{enumerate}
We denote this collection of leaves by $(e_1,\dots, e_p)$ labeled by proximity to the triangle $t_0$. By convention, we assume that $e_1$ is the leaf with endpoints $x_1,x_3$ and $e_p$ is the edge of $\gamma(t_0)$ closest to $e_1$. Let us set $e_0=\gamma^{-1}(e_p)$ and $e_{p+1}=\gamma(e_1)$, and 
\[
\cal E=\cal E_{\gamma,p}=(e_0,\dots, e_{p+1}).
\]
The vertices of the leaves in $\cal E$ determine a \emph{finite} cyclically ordered list of points $(x_1,\dots x_k)$ in $\bdry \wt S$ which defines a polygon $\Pi_{\gamma,p}$. Note that $\cal E$ is a disconnecting triangulation for $\Pi_{\gamma,p}$ in the sense of \S \ref{ssec:finitepoly}. Observe that the set $\cal T_{\cal E}$ of triangles defined by $\cal E$ is not a subset $\cal T_{\cal \lambda}$ as soon as one of the diagonals of $\cal E$ is the lift of a closed leaf of $\wt \lambda$.

\subsection{A technical result for $m=3$}\label{ssec:dim3}

Let $\wt c$ be a lift of an element of $\lambda^+_c$. A choice of a lift $\wt k$ of an arc in $\cal K$ corresponding to $\wt c$ defines three vertices $x_{c^+},x_{c^-},x_{c^r}$ in $\bdry \wt S$ as in \S \ref{ssec:coords}.

In this section, we prove the assumptions of Theorem \ref{thm:action}.\ref{cond:d3} give us control over the unique triple ratio of the positive triple of flags $(\xi_{\rho_n}(x_{c^+}), \xi_{\rho_n}(x_{c^-}), \xi_{\rho_n}(x_{c^r}))$. More precisely, we prove the following.

\begin{thm}\label{thm:d3} Let $(\rho_n)$ be a sequence of Hitchin representations of tree-type in $\sf{Hit}_3(S)$. Let $c$ be a closed leaf for the maximal geodesic lamination $\lambda^+$ and suppose there exists $C>0$ such that 
\[
\log\frac{\lambda_a}{\lambda_{a+1}}(\rho_n,c)>C
\]
for all $n>1$ and for $a=1,2$. Consider the points $x_{c^+},x_{c^-},x_{c^r}\in\bdry \wt S$ corresponding to a lift $\wt c$ of $c$ to the universal cover $\wt S$ and a corresponding lift $\wt k$ of the transverse arc $k$. 
Then,
\begin{align}\label{eqn:tripleratiod3}
\lim_{n\to\infty} r_n\log(T(\xi_{\rho_n}(x_{c^+}),\xi_{\rho_n}(x_{c^-}),\xi_{\rho_n}(x_{c^r}))=0
\end{align}
\end{thm}
\begin{remark} The assumption $m=3$ in Theorem \ref{thm:d3} is due to an explicit computation required for our proof.
\end{remark}

\begin{proof} Denote by $\gamma\in\pi_1(S)$ the group element corresponding to $\wt c$. Let us ease notation by setting
\begin{gather*}
X_n=T(\xi_{\rho_n}(x_{c^+}),\xi_{\rho_n}(x_{c^-}),\xi_{\rho_n}(x_{c^r})),\ Y_n=T(\rho_n(c)\xi_{\rho_n}(x_{c^r}),\xi_{\rho_n}(x_{c^-}),\xi_{\rho_n}(x_{c^r})),\\
\mu_{a,n}=\frac{\lambda_a}{\lambda_{a+1}}(\rho_n,c)
\end{gather*}
Note that $\lim_nr_n\log Y_n=0$ by Theorem \ref{thm:FLP}. In fact, the points $\gamma.x_{c^r},x_{c^-},x_{c^r}$ determine an ideal triangle $t$ which is either in $\cal T_{\lambda^+}$ or it is a triangle in an ideal triangulation which differs from $\lambda^+$ by a diagonal flip. For $\veps>0$, let $N>0$ be such that for all $n>N$
\[
\exp\left(-\frac{\veps}{r_n}\right)<Y_n<\exp\left(\frac{\veps}{r_n}\right).
\]
We express $X_n$ in terms of $Y_n$. For every $n$, let $(u_{1,n},u_{2,n},u_{3,n})$ be the basis of $\bb R^3$ such that
\begin{align*}
\xi_{\rho_n}(x_{c^+})&= \rm{Span}(u_{1,n})\subset \rm{Span}(u_{1,n},u_{2,n}),\\
\xi_{\rho_n}(x_{c^+})&= \rm{Span}(u_{3,n})\subset \rm{Span}(u_{3,n},u_{2,n}),\\
\xi_{\rho_n}(x_{c^r})^{(1)}&=\rm{Span}(u_{1,n}+u_{2,n}+u_{3,n}).
\end{align*}
With respect to this basis
\begin{align*}
\xi_{\rho_n}(x_{c^r})^{(2)}&=\xi_{\rho_n}(x_{c_r})^{(1)}\oplus \rm{Span}(u_{2,n}+\left(1+X_n\right)u_{3,n}),\\
\rho(c)\xi_{\rho_n}(x_{c^r})^{(1)}&=\rm{Span}\left(u_{1,n}+\frac{1}{\mu_{1,n}}u_{2,n}+\frac{1}{\mu_{1,n}\mu_{2,n}}u_{3,n}\right),\\
\rho(c)\xi_{\rho_n}(x_{c^r})^{(2)}&=\rho_n(c)\xi_{\rho_n}(x_{c_r})^{(1)}\oplus\rm{Span}\left(u_{2,n}+\frac{1}{\mu_{2,n}}\left(1+X_n\right)u_{3,n}\right).
\end{align*}
Thus, we can compute explicitly
\[
Y_n=\frac{\left(1-\mu_{2,n}\right)-\left(\mu_{1,n}-1\right)\mu_{2,n}X_n}{\left(1-\mu_{1,n}\right)X_n-\mu_{1,n}(\mu_{2,n}-1)}, \text{ and }
X_n=\frac{\left(1-\frac{1}{\mu_{2,n}}\right)\left(Y_n+\frac{1}{\mu_{1,n}}\right)}{\left(1-\frac{1}{\mu_{1,n}}\right)\left(1+\frac{Y_n}{\mu_{2,n}}\right)}
\]
Note that, by hypotheses, there exists $D>0$ such that the following inequalities hold
\[
0<\frac{1}{\mu_{a,n}}<1,\ \ \ \ \ \ \ \ \ \frac{D}{1+D}<1-\frac{1}{\mu_{a,n}}<1.
\]
Thus
\begin{equation}\label{eqn:endof3}
\begin{aligned}
X_n&>\frac{\left(\frac{D}{1+D}\right)Y_n}{1+Y_n}>\left(\frac{D}{1+D}\right)\frac{\exp\left(-\frac{\veps}{r_n}\right)}{1+\exp\left(\frac{\veps}{r_n}\right)}\\
X_n&<\frac{1+Y_n}{\frac{D}{1+D}}<\left(1+\frac{1}{D}\right)\left(1+\exp\left(\frac{\veps}{r_n}\right)\right)
\end{aligned}
\end{equation}
Recall that we assume $\lim_nr_n=0$. Combining the Inequalities \ref{eqn:endof3} with Equation \ref{eqn:useful} for $C=\exp\left(\veps\right)>1$, we get
\[
-\veps \leq \lim_nr_n\log X_n\leq \veps
\]
and this concludes the proof as $\veps$ was arbitrary.
\end{proof}

\subsection{Estimates for the Hilbert length along a tree-type sequences}\label{ssec:proofaction}

We replace the Hilbert length with an algebraic quantity better adapted to our setup. Denote by $\rm{Tr}(A)$ the trace of a matrix $A$. Observe that the function 
\begin{align*}
T_H\colon \rm {GL}_m(\bb R)&\to \bb R\\
A&\mapsto \rm {Tr}(A)\rm {Tr}(A^{-1})
\end{align*}
is invariant under scaling and induces a function $T_H\colon\rm {PSL}_m(\bb R)\to\bb R$.

\begin{lem}\label{lem:trace} Let $c$ be a non-trivial conjugacy class in $\pi_1(S)$. Let $(\rho_n)$ be a sequence of Hitchin representations. Suppose that $\lim_n\ell_H(\rho_n,c)$ diverges. Then,
\[
\lim_n \frac{\log T_H(\rho_n(c))}{\ell_H (\rho_n, c)}=1.
\]
\end{lem}
\begin{proof} By Proposition \ref{prop:gtilde}, there exists a representative $A_{n,c}$ of $\rho_n(c)$ to $\rm{GL}_m(\bb R)$ which is diagonalizable over $\bb R$ with positive distinct eigenvalues: $\lambda_{1,n}>\dots>\lambda_{m,n}$.
This implies that,
\[
\frac{\lambda_{1,n}}{\lambda_{m,n}}< T_H(A_{n,c})=\frac{\lambda_{1,n}}{\lambda_{m,n}}\left(1+\frac{\lambda_{2,n}}{\lambda_{1,n}}+\dots+\frac{\lambda_{m,n}}{\lambda_{1,n}}\right)\left(1+\frac{\lambda_{m,n}}{\lambda_{m-1,n}}+\dots+ \frac{\lambda_{m,n}}{\lambda_{1,n}}\right)\leq \frac{\lambda_{1,n}}{\lambda_{m,n}}\cdot m^2
\]
As $\ell_H (\rho_n, c)=\log\frac{\lambda_{1,n}}{\lambda_{m,n}}$, it follows that
\[
1 \leq \lim_n \frac{\log T_H(\rho_n(c))}{\ell_H (\rho_n,c)}\leq \lim_n \left(1+\frac{\log m^2}{\ell_H(\rho_n,c)}\right)=1.\qedhere
\]
\end{proof}

Let $c$ be a non-trivial conjugacy class in $\pi_1(S)$. Choose $\gamma\in c$, and $p$ on the axis of $\gamma$. By \S \ref{ssec:finipoly2} this data determines a finite polygon $\Pi_{\gamma,p}$ inscribed in $\bdry S$ equipped with a disconnecting triangulation $\cal E=\cal E_{\gamma,p}$. 

Let $t_0$ be the triangle $\Pi_{\gamma,p}$ containing $p$ with vertices $x_1,x_2,x_3$. Let $y_1,y_2,y_3$ be vertices of $\Pi_{\gamma,p}$ such that $\gamma(x_i)=y_i$. The limit maps $\xi_{\rho_n}$ give us sequences of flags
\[
F_{i,n}=\xi_{\rho_n}(x_i),\text{ and }G_{i,n}=\xi_{\rho_n}(y_i)=\rho_n(\gamma)\cdot \xi_{\rho_n}(x_i).
\]
In particular, Propositions \ref{prop:gtilde}-\ref{prop:multi} gives a specific sequence of totally positive matrix $A_{n,c}$ in $\rm{GL}_m(\bb R)$ which represents $\rho_n(\gamma)\in\rm{PSL}_m(\bb R)$. We wish to explicitly understand the asymptotic behavior of the sequence $\log T_H(A_{n,c})$ assuming the hypotheses of Theorem \ref{thm:action}.

\begin{thm}\label{thm:actionsuper} Let $\rho_n$ be a sequence of Hitchin representations of tree-type. Let $c$ be a non-trivial conjugacy class in $\pi_1(S)$ such that
\[
\lim_nr_n\ell_H (\rho_n, c)\in [0,\infty).
\]
and either of the following assumptions holds.
\begin{enumerate}[(a)]
	\item The curve $c$ has zero intersection number with the closed leaves in $\lambda^+_c$.
	\item The sequence $(\rho_n)$ is contained in $\sf{Hit}_3(S)$, $\lim_nr_n\sigma^a(\rho_n,d)\geq 0$ for all $a=1,2$, and $d\in\lambda^+_c$, and there exists $C>0$ such that for all $a=1,2$, and $d\in\lambda_c$
	\[
	\log\frac{\lambda_a}{\lambda_{a+1}}(\rho_n,d)>C.
	\] 
\end{enumerate} 
Denote by $\cal E=\cal E_{\gamma,p}$ the disconnecting triangulation of the polygon $\Pi_{\gamma,p}$ described in \S\ref{ssec:finipoly2}. Then, if $A_{n,c}$ is the representative of $\rho_n(c)$ given by Proposition \ref{prop:multi}
\[
\lim_nr_n\log\rm{Tr}(A^{-1}_{n,c})=0,\text{ and }\lim_n r_n\log\rm{Tr}(A_{n,c})=\lim_nr_n\sum_{e\in\cal E}\sigma(\rho_n, e).
\]
\end{thm}
\begin{proof}

Given a matrix $M$, let us denote by $M_{ij}$ its $(ij)$-th entry. Applying Proposition \ref{prop:multi} to our setup we can write
\[
A_{n,c}=\frak D_{n,1}\frak T_{n,1}\cdots \frak D_{n,k}\frak T_{n,k}.
\]
 As $A_{n,c}$ is totally positive, its diagonal entries are positive and $\rm{Tr}(A_{n,c})\geq (A_{n,c})_{mm}$. Furthermore, since for every $i=1,\dots, k$, the entries $(\frak D_{n,i})_{mm},(\frak T_{n,i})_{mm}$ are positive, we have
\begin{equation}\label{eqn:trace1}
\begin{aligned}
(A_{n,c})_{mm}&\geq (\frak D_{n,1})_{mm}\cdots (\frak T_{n,k})_{mm}\\
&=\prod_{e\in\cal E}\exp(\sigma(\rho_n, e))\prod_{i=1}^{k}(\frak T_{n,k})_{mm}
\end{aligned}
\end{equation}

By Definition \ref{def:treetype}, Theorem \ref{thm:FLP}, and Theorem \ref{thm:d3}, for $\veps>0$ there exist $L$ such that for all $n>L$ the following holds
\begin{itemize}
	\item $\exp\left(-\frac{\veps}{r_n}\right)<T^{abc}(\rho_n,t)<\exp\left(\frac{\veps}{r_n}\right)$ for all $t\in\cal T_{\cal E}$, for all $a,b,c\in\bb Z_{>0}$ with $a+b+c=m$.
	\item $D^{a}(\rho_n,e)>\exp\left(-\frac{\veps}{r_n}\right)$ for all $e\in\cal E$ and for all $a=1,\dots, m-1$.
\end{itemize}
In particular, it follows that $(\frak D_{n,i})_{mm}\exp\left(\frac{m\veps}{r_n}\right)$ is an upper bound for all the entries of the matrix $\frak D_{n,i}$. Moreover, for each $i=1,\dots, k$ and $j,l=1,\dots, m$ there exists a constant $S_{ijl}>0$ such that
\[
(\frak T_{n,i})_{j,l}<S_{ijl}\exp\left(\frac{\veps}{r_n}\right).
\]
Set $S=\max_{i,j,l}S_{i,j,l}>0.$ Then, using once again the positivity of the entries of $A_{n,c}$, for every $j=1,\dots, m$
\begin{equation}\label{eqn:trace2}
\begin{aligned}
(A_{n,c})_{jj}&=\sum_{j_1,\dots, j_{k-1}}(\frak D_{n,1})_{jj}(\frak T_{n,1})_{jj_1}\dots (\frak D_{n,k})_{j_{k-1}j_{k-1}}(\frak T_{n,k})_{j_{k-1}j}\\
&<S^{k}\exp\left(\frac{k\veps}{r_n}\right)\prod_{i=1}^{k} (\frak D_{n,i})_{mm}\exp\left(\frac{m\veps}{r_n}\right)\\
&=S^{k}\exp\left(\frac{(m+1)k\veps}{r_n}\right)\prod_{e\in\cal E}\exp(\sigma(\rho_n, e))
\end{aligned}
\end{equation}
Combining Inequalities \ref{eqn:trace1} and \ref{eqn:trace2} we see that
\[
\lim_n r_n\log\rm{Tr}(A_{n,c})=\lim_nr_n\sum_{e\in\cal E}\sigma(\rho_n, e).
\]

We now focus on $\rm{Tr}(A_{n,c}^{-1})$. Note that $A_{n,c}^{-1}$ is not a totally positive matrix. However, let us set $P$ equal to the diagonal matrix 
\[
\begin{bmatrix}
-1&&&\\
&1&&\\
&&\ddots&\\
&&&(-1)^m
\end{bmatrix}.
\]
Observe that $P^2$ is the identity matrix,  $P$ commutes with diagonal matrices, and that $PA_{n,c}^{-1}P$ is totally positive by Cramer's rule. Thus
\begin{align*}
PA_{n,c}^{-1}P&=P\frak T^{-1}_{n,k}\frak D^{-1}_{n,k}\dots \frak T^{-1}_{n,1}\frak D^{-1}_{n,1}P\\
&=(P\frak T^{-1}_{n,k}P)(P\frak D^{-1}_{n,k}P)\dots (P\frak T^{-1}_{n,1}P)(P\frak D^{-1}_{n,1}P)
\end{align*}
and the right hand side is a product of totally positive matrices $\frak U_{n,i}=P\frak T^{-1}_{n,i}P$ and $\frak E_{n,i}=\frak D^{-1}_{n,i}$.

In order to prove that $\lim_{n}r_n\log \rm{Tr}(A^{-1}_{n,c}=0$ we repeat, with minor variations, the first part of the proof. In summary, we start by observing that for every $\veps>0$, there exist $L>0$ and a positive constant $S'$ such that for $n>L$
\begin{enumerate}
	\item $(\frak U_{n,i})_{jl}<S'\exp\left(\frac{\veps}{r_n}\right)$, and
	\item $(\frak E_{n,i})_{jj}<\exp\left(\frac{m\veps}{r_n}\right)$.
\end{enumerate}
Moreover, by total positivity $(PA_{n,c}^{-1}P)_{11}$ is larger than
\[
(\frak D_{n,1})_{11}\dots (\frak T_{n,k})_{11}=\prod_{i=1}^{k}(\frak T_{n,k})_{11}>\exp\left(-\frac{\veps}{kr_n}\right).
\]
The proof follows by combining these inequalities.
\end{proof}

Theorem \ref{thm:action} follows at once from Theorem \ref{thm:actionsuper}.

\begin{remark} The proof of Theorem \ref{thm:actionsuper} shows that one could generalize Theorem \ref{thm:action}.\ref{cond:d3} to $m\geq 3$ by assuming the appropriate generalization of Equation \ref{eqn:tripleratiod3}.
\end{remark}

\appendix
\section{Scaling sequences from geodesic currents}\label{sec:scaling}

In this brief appendix we provide an example of a choice of scaling sequence for a sequence of Hitchin representations arising in the context of geodesic currents. Lemma \ref{lem:currents} was proved together with C. Ouyang. M. B. Pozzetti independently communicated the proof to the author.

We start by recalling the necessary results from the theory of geodesic currents. We refer to \cite{Bon86,Bon88} for more background.

A \emph{geodesic current} $\mu$ is a $\pi_1(S)$-invariant, locally finite, Borel measure on the set of geodesics $\cal G(\wt S)$ of $\wt S$. For example every conjugacy class $c$ of elements in $\pi_1(S)$ defines a geodesic current by considering the atomic measure on the axes of representatives of $c$. The space $\cal C(S)$ of geodesic current on $S$ is a cone in an infinite dimensional vector space and its projectivization $\cal {PC}(S)$ is compact. The \emph{intersection pairing} is a continuous, symmetric, bilinear, positive function $i\colon \cal C(S)\times\cal C(S)\to \bb R^+$ which extends the intersection number of curves to the space of geodesic currents.

Zhang and the author \cite{MZ19}, generalizing work of Bonahon \cite{Bon88}, associate to every Hitchin representation $\rho$ an \emph{intersection geodesic current} $\mu_\rho$ such that for each conjugacy class $c$ in $\pi_1(S)$
\[
\ell_H(\rho,c)=i(\mu_\rho,c).
\]
See \cite{BCLS18} for an independent construction of geodesic currents associated to Hitchin representations.

Consider a sequence $\rho_n$ of Hitchin representations and the corresponding sequence of geodesic currents $\mu_{\rho_n}$. By compactness of $\cal{PC}(S)$, there exists a sequence of positive real numbers $r_n$ such that $r_n\mu_n$ converges to $\mu\in\cal C(S)$. We want to observe that $\lim_n r_n=0$ and, thus, $r_n$ is a scaling sequence in the sense of Definition \ref{def:treetype}. It suffices to show the following.

\begin{lem}[M.-Ouyang, Pozzetti]\label{lem:currents} The map from $\sf {Hit}_m(S)\to \cal C(S)$ is proper.
\end{lem}
\begin{proof} Let $\alpha$ be finite collection of simple closed curves on $S$ such that
\begin{enumerate}[-]
	\item $\alpha$ contains the curves in a pants decomposition of $S$;
	\item the complement of $\alpha$ in $S$ is a union of topological discs.
\end{enumerate} 
By \cite[Lemma 4]{Bon88}
\[
\bigcup_{p>0}\{\mu\in \cal C(S)\colon i(\mu,\alpha)\leq p\}
\]
is an exhaustion of $\cal C(S)$ by compact sets. By \cite[Corollary 1.5]{LZ17}, the map
\[
\sf{Hit}_m(S)\to \bb R^{k}
\] 
that sends $\rho$ to the Hilbert lengths of the curves in $\alpha$ is proper. Therefore, if $\rho_n$ diverges in $\sf{Hit}_m(S)$, then $i(\mu_n,\alpha)$ diverges and, thus $\mu_n$ diverges in $\cal C(S)$.
\end{proof}
\bibliography{mybib}
\bibliographystyle{abbrv}
\end{document}